\documentclass[lettersize,journal]{IEEEtran}

\usepackage{amsfonts,comment,amsthm}
\usepackage{amssymb,amsmath,url,bm}
\usepackage{mathrsfs}
\usepackage{hyperref}
\usepackage{graphicx,color,float,epstopdf}

\usepackage{pst-blur,pstricks-add}
\usepackage{cases}
\usepackage{algorithm}
\usepackage{algorithmicx}
\usepackage{algpseudocode}
\usepackage{amsfonts, bm}
\usepackage{graphics,booktabs,color,epsfig,subfigure}
\usepackage[numbers,sort&compress]{natbib}
\usepackage{amsmath,amssymb,amsfonts}
\usepackage{graphicx}
\usepackage{textcomp}
\usepackage{xcolor}
\usepackage{subfigure}
\usepackage{multirow}

\DeclareMathOperator{\prox}{Prox}
\newtheorem{lemma}{Lemma}
\newtheorem{proposition}{Proposition}
\newtheorem{theorem}{Theorem}
\newtheorem{Def}{Definition}

\newtheorem{remark}{Remark}
\newcommand{\h}[1]{\mathbf{#1}}
\DeclareMathOperator{\sign}{sign}

\DeclareMathOperator{\dist}{dist}

\newcommand{\nn}{\nonumber}

\begin{document}
\bibliographystyle{IEEEtran}
\title{On Partly Smoothness, Activity Identification and  Faster Algorithms of  $L_1$ over $L_2$  Minimization}

\author{ Min Tao, Xiao-Ping  Zhang,~\IEEEmembership{Fellow,~IEEE} and Zi-Hao Xia
\thanks{This work of M. Tao is supported partially by the Natural Science Foundation of China (No.
11971228) and partially by
Jiangsu University QingLan Project. The work of X.-P. Zhang is supported by the Natural Sciences and
Engineering Research Council of Canada (NSERC), Grant No. RGPIN-2020-04661.}
\thanks{M. Tao is with the Department of Mathematics, National Key Laboratory for Novel Software Technology, Nanjing
University, Nanjing 210093, China. Email: taom@nju.edu.cn}
\thanks {X.-P. Zhang is with the Department of Tsinghua-Berkeley Shenzhen Institute, Tsinghua University, Beijing, China,
Department of Electrical, Computer and Biomedical Engineering, Toronto Metropolitan University,
Toronto ON M5B 2K3, Canada. Email: xpzhang@ieee.org}
\thanks{Z.-H. Xia is  with the Dept. of
   Mathematics,  Nanjing University.}
}

\maketitle

\begin{abstract}
The $L_1/L_2$ norm ratio arose as a sparseness measure
and attracted a considerable amount of attention due to  three merits: (i) sharper approximations of $L_0$ compared to the $L_1$;
(ii) parameter-free and scale-invariant; (iii) more attractive than $L_1$ under highly-coherent matrices.

In this paper, we first establish the partly smooth property of $L_1$ over $L_2$ minimization relative to an active manifold ${\cal M}$ and also demonstrate its prox-regularity property. Second, we reveal that ADMM$_p$ (or ADMM$^+_p$) can identify the active manifold within a finite iterations. This discovery contributes to a deeper understanding of the optimization landscape associated with $L_1$ over $L_2$ minimization.
Third, we propose a novel heuristic algorithm framework that combines ADMM$_p$ (or ADMM$^+_p$) with a globalized semismooth Newton method tailored for the active manifold ${\cal M}$. This hybrid approach leverages the strengths of both methods to enhance convergence.
Finally, through extensive numerical simulations, we showcase the superiority of our heuristic algorithm over existing state-of-the-art methods for sparse recovery.

\end{abstract}

\begin{IEEEkeywords}
	 Sparse recovery,  partly smooth, Prox-regularity, active set, nonsmooth analysis, identifiable surface
\end{IEEEkeywords}

\maketitle
	\section{Introduction}

In scientific and engineering domains, the primary goal is to acquire a low-dimensional representation from high-dimensional data. This objective heavily relies on the critical assumption of sparsity. Sparsity has proven to be essential in effectively addressing diverse applications across fields such as compressive sensing \cite{5466604,WenYinGoldZhang10,Shi2010AFH}, machine learning \cite{hastie2009elements}, image processing \cite{KTF11,Audrey15}, and matrix factorization \cite{8653529}.

The use of the $L_1$ norm over $L_2$ norm in nonconvex regularized models has gained increasing attention in the field of compressive sensing \cite{ELX13, Audrey15,YEX14, Tao20, TaoZhang23,XU2021486,ZengYuPong20,RWDL19,PHAM201796}. This ratio is known for its ability to promote sparsity, and exhibits desirable properties such as scale-invariance \cite{5238742} and parameter-free characteristics and is much more efficient than $L_1$ when  sensing matrix
highly-coherent \cite{DE03,GN03}. More discussion can be found in \cite{YEX14,RWDL19}.
 As a result, $L_1/L_2$ minimization has found widespread applications in various domains; see, e.g., \cite{6638838, KTF11, Audrey15,PHAM201796, JIA2018198}.

Much of the existing literature focuses on the unconstrained,
\begin{eqnarray}
\label{L1o2uncon}
\min _{{\h x}\in {\cal X}} F_{u}({\h x}):=\gamma\frac{\|{\h x}\|_1}{\|{\h x}\|_2}+\Phi({\h x}),
\end{eqnarray}
where
 ${\cal X}=\mathbb R^n/{\mathbb R}^n_+$ (resp. arbitrary signal and nonnegative signal) and the data-fitting term of $\Phi:{\mathbb R}^m \rightarrow {\mathbb R}$ differentiable and
 the balance parameter $\gamma>0$.
In this paper, we focus on the model (\ref{L1o2uncon}).

The equivalence between $L_1/L_2$ and $L_0$ for signal reconstruction under different conditions has been studied by Yin et al. \cite{YEX14}, Xu et al. \cite{XU2021486}, and Zhou and Yu \cite{ZHOU2019247}. However, due to the nonconvex and nonsmooth nature of $L_1/L_2$ minimization, finding a global solution is challenging. Consequently, various algorithms have been developed to search for stationary points.
Specifically,  the scaled gradient projection method (SGPM) \cite{ELX13, YEX14} has been proposed for (\ref{L1o2uncon}) with ${\cal X}={\mathbb R}_+^n$.
The alternating direction method of multipliers (ADMM) \cite{RWDL19} with double variable duplication and accelerated schemes \cite{WYYL20} has been proposed for the following constrained model (\ref{L1o2con}) with $\sigma=0$ for noiseless data.

\begin{eqnarray}
\label{L1o2con}
\min _{{\h x}\in {\cal X}} F_c({\h x}):=\frac{\|{\h x}\|_1}{\|{\h x}\|_2}+\iota_{\{{\h x}:\|A{\h x}-{\h b}\|_2\le \sigma\}}(\h x).
\end{eqnarray}

The constrained model (\ref{L1o2con}) aims to address the recovery from observations corrupted by Gaussian noise, where $\sigma$ corresponds to the noise level. Specifically, $\sigma=0$ indicates noise-free observations, while $\sigma>0$ is associated with noisy data. The moving-balls-approximation-based algorithms \cite{ZengYuPong20} have been developed to address the constrained model (\ref{L1o2con}) in both noise-free and noisy scenarios. The intricate connections between (\ref{L1o2uncon}) and (\ref{L1o2con}), we refer to \cite{ZengYuPong20,Chen16} for more discussions.
 For solving (\ref{L1o2uncon}), the ADMM$_p$ \cite{Tao20} and ADMM$_p^+$ \cite{TaoZhang23} methods have been introduced in our previous work. These approaches come with global convergence guarantees and are tailored for cases where $\mathcal{X}=\mathbb{R}^{n}$ and $\mathcal{X}=\mathbb{R}_+^{n}$, respectively.
Additionally, an alternating forward-backward algorithm with a smoothing technique (SOOT) was proposed in \cite{Audrey15} for solving
 blind deconvolution  via $L_1$ over $L_2$ minimization.

 The superiority of ADMM$_p$ and ADMM$_p^+$ has been empirically validated in the context of sparse arbitrary/nonnegative signal recovery \cite{Tao20, TaoZhang23}. Notably, the explicit expression for all proximal points of the proximal operators of $L_{1}/L_{2}$ has been characterized in \cite[Theorem 3.2]{Tao20}. Furthermore, in \cite[Theorem 3.3]{Tao20} and \cite[Theorem 6]{TaoZhang23}, one of the proximal points of the proximal operators $L_{1}/L_{2}$ (i.e., $\frac{\|{\h x}\|_1}{\|{\h x}\|_2}$) and $(L_{1}/L_{2})^+$ (i.e.,
$\frac{\|{\h x}\|_1}{\|{\h x}\|_2}+\iota_{{\mathbb R}^n_+}({\h x})$), respectively, has been derived without any dependence on unknown parameters. These characterizations have paved the way for the development of practical solvers that compute exact solutions for these $L_{1}/L_{2}$ proximal operators, as demonstrated in algorithms; see, e.g., \cite[Algorithm 3.1]{Tao20} and \cite[Algorithm 1]{TaoZhang23}.

\subsection{Motivation}
The growing scale of datasets has prompted extensive research into uncovering the inherent sparsity pattern
 of (\ref{L1o2uncon}) during the early phases of iterative processes.
Proximity operators have emerged as a valuable tool for detecting sparsity patterns within iterates, as highlighted in studies such as \cite{LW04,VSGMFP15, LFP14,LW16}. A prominent example is the widely used $L_1$-norm. The associated proximity operator, known as soft-thresholding \cite{Dono95}, aids in identifying the support set \cite{LFP14,HYZ} under some non-degenerate conditions. This property, referred to as the identification property, has been extensively explored in both convex and nonconvex contexts, as demonstrated in works by \cite{JJ88, Burke1990OnIO, Wright93,Aris06,Aris14,7944626}, and others.

It is important to clarify that the identification property does not prescribe a specific algorithm but rather outlines the fundamental properties required. Several algorithms fulfill these criteria, including projected gradient methods \cite{1101194}, projected Newton methods \cite{doi:10.1137/0320018}, and forward-backward-type methods \cite{Liang17}.
A closely related concept is the notion of partly smoothness \cite{Lewis02} (see Definition \ref{def2.3}). Partly smoothness requires a function to exhibit smoothness along an active manifold while demonstrating sharp growth in directions away from that manifold. Building upon this concept, \cite{LW04} demonstrated that if a ${\cal C}^p$-partly smooth function ($p\ge 2$) (see \cite[Definition 2.3]{LW04}) is prox-regular (see Definition \ref{defprox}), the corresponding active manifold must be unique. Under non-degeneracy conditions and certain mild assumptions, \cite{LW04} further established that the adopted algorithm can identify the active manifold within a finite number of steps.
It is important to emphasize that partly smoothness is a property inherent to the studied problem, while the identification property is typically associated with the chosen algorithm.

In the present paper, we aim to address the following questions, ``Whether the  $L_1/L_2$ minimization have the
partly smooth property?" and ``Can ADMM$_p$ (or ADMM$_p^+$) identify the optimal manifold in a finite number of iterations when applied to (\ref{L1o2uncon})?"  We answer both questions affirmatively.

\subsection{This paper}
Theoretically, we establish the partly smooth and prox-regular properties of $L_1/L_2$ minimization. We then prove general results for identifying the active manifold in a finite number of iterations, relaxing conditions compared to the classical result outlined in \cite[Theorem 5.3]{LW04}.

Specifically, we extend the scope by relaxing the assumption from ``the function $f$ being ${\cal C}^p$-partly smooth ($p\ge2$)" as stated in \cite[Theorem 5.3]{LW04} to ``$f$ being partly smooth, i.e., ${\cal C}^1$ (see Definition \ref{def2.3})." Consequently, we demonstrate the identification property of ADMM$_p$ \cite{Tao20} and ADMM$_p^+$ \cite{TaoZhang23} under relatively mild conditions.

The finite identification property signifies the existence of a specific index number $K$. The sequence generated by either ADMM$_p$ or ADMM$_p^+$ enters the active manifold when the iteration number goes beyond $K$. Once the active manifold is identified, the optimization problem transforms into a low-dimensional minimization, allowing the use of high-order algorithms tailored for low-dimensional manifolds to significantly enhance convergence speed.

Efficiency of ADMM$_p$ or ADMM$_p^+$ for sparse recovery has been well-illustrated \cite{Tao20, TaoZhang23}, outperforming existing state-of-the-art methods such as SGPM \cite{YEX14, ELX13}, GIST (\cite[Algorithm 1]{GZLHY}), and monotone/nonmonotone APG \cite{LinLi15}, as well as SOOT \cite{Audrey15}. Therefore, leveraging ADMM$_p$ or ADMM$_p^+$ as an initial phase and integrating a globalized semismooth Newton method on the active manifold will further accelerate the performance.

We clarify that the reasons for adopting a globalized semismooth Newton method are: (i) The objective function in (\ref{L1o2uncon}) is ${\cal C}^1$ (continuously differentiable) rather than ${\cal C}^2$ (twice continuously differentiable) when restricted to the active manifold; (ii) There is no guarantee that the objective function remains convex when restricted to the active manifold, even if the manifold is a space.

In addition, although we know that the iterates will eventually identify the active manifold within a finite number of steps, an explicit expression for the index $K$ (enter the active manifold) is unavailable. This prompts us to employ heuristic techniques for active manifold identification. More specifically, we introduce a two-phase heuristic acceleration framework:
\begin{itemize}
\item Initially, either ADMM$_p$ or ADMM$_p^+$ is employed to identify the active manifold by comparing adjacent support sets. The support sets remain unchanged over a specified number of iterations denoted as $T$ (where $T \in \mathbb{N}_+$).

\item  In the second phase, we apply a globalized semismooth Newton method on the active manifold to achieve rapid convergence.
\end{itemize}
It's worth mentioning that there is no theoretical guarantee for the heuristic strategy to identify the active manifold; however, we observe that it practically performs very well. With the identified active manifold, the semismooth Newton method is incorporated to solve a lower-dimension problem. Finally, we establish the eventual superlinear/quadratic convergence rate of the approach under some mild conditions.

We conducted extensive experiments to analyze algorithmic behaviors and compare various sparse recovery models. Our tests encompassed sparse recovery scenarios involving highly coherent compressive matrices using synthetic data, as well as real datasets. Our newly proposed heuristic acceleration algorithm demonstrated the ability to identify the active manifold within a finite number of steps. Notably, it significantly outperforms other state-of-the-art methods for $L_1$ over $L_2$ minimization problem.

\subsection{Contribution}
\begin{enumerate}
	
	\item {\bf Partly smoothness and prox-regularity of $L_1$ over $L_2$.} We prove the {\it partly smooth} and {\it prox-regular} properties of the $L_1/L_2$ model (\ref{L1o2uncon}), which are the theoretical basis for establishing the finite identification of  ADMM$_p$ \cite{Tao20} and ADMM$_p^+$ \cite{TaoZhang23}.

\item {\bf Support detection.}
ADMM$_p$ (or ADMM$_p^+$) is shown to identify the active manifold in finite steps, known as {\it support detection} in compressive sensing.
 We  relax the conditions in the classical results \cite[Theorem 5.3]{LW04} by
 replacing   ${\cal C}^p$-partly smooth ($p\ge 2$) with  partly smooth (i.e., ${\cal C}^1$), broadening the applicability of ADMM$_p$ (or ADMM$_p^+$) for finite identification.
This expansion also clarifies our choice of the semismooth Newton method on the active manifold rather than the standard Newton algorithm.
	\item {\bf Fast algorithm with high accuracy.}
A heuristic acceleration framework is introduced, employing ADMM$_p$/ADMM$_p^+$ for active manifold identification and the globalized semismooth Newton method for convergence acceleration. Theoretical analysis demonstrates superlinear/quadratic convergence rates under mild conditions within the active manifold.
Extensively numerical simulations verify its high-order convergence rate.

\end{enumerate}

\subsection{Organization}
The rest of this paper proceeds as follows. We describe the notations and definitions in Section II and present a lemma to
characterize the prox-regularity of
fraction function.
 Section III elaborates on verifying the partly smoothness property of $L_1$ over $L_2$ minimization problem.
Section IV establishes the general results for identifying the active manifold by relaxing
the conditions in the classical result \cite[Theorem 5.3]{LW04}. Equipped with this, we
 validate ADMM$_p$ \cite{Tao20} (or ADMM$_p^+$ \cite{TaoZhang23}) with the property of identifying the active manifold in finite iterations under much weaker conditions and along with non-degenerate assumption. Furthermore, we exploit sufficient conditions to ensure the non-degenerate condition.
Additionally, we introduce a heuristic acceleration framework on the active manifold in Section IV to achieve a faster convergence rate for the $L_1$ over $L_2$ minimization problem.
In Section V, we present extensive experiments to showcase the superior performance of our proposed heuristic acceleration framework for sparse recovery, using both synthetic and real datasets.
Finally, our conclusions are presented in Section VI.

\section{Preliminary} \label{Sec-Preliminaries}

 The bold letter denotes the vector. $x_{i}$  and $|{\h x}|$ denote the $i$-th entry and the absolute value vector of
${\h x}$ in entrywise.
  $\|{\h x}\|_p$  denote its
 $p$-norm (where $p=1,2,1/2$).
  $\sign(\h x)$ is defined as a vector of the same length as $\h x$ with its $i$-component equal to
   the set $[-1,1]$ if ${x}_i=0;$ otherwise being the sign of each component of ${\h x}$.
  ${\text{supp}}({\h x})$ denotes its support set.
  The notation of ${{\mathbb R}^n_+}/{\mathbb R}^n$ means ${{\mathbb R}^n_+}$ or ${\mathbb R}^n$.
Let ${{\h e}}$ and ${\h e}_i$ be the vectors with all entries equal to $1$ and the $i$th entry equal to $1$ while the others are zero, respectively.
   $I_{n}$ is $n \times n$ identity matrix.
Given a set of ${\mathcal D}$,
 ${\rm co}{\mathcal D}$, $\mathcal{D}^c$ and ${\text{\rm ri}}({\cal D})$  denote its  convex hull, its complementary and its relative interior.
  If $\mathcal{D} \subseteq [n]$ ($[n]:=\{1,\cdots,n\}$), we use $\sharp(\mathcal{D})$  to present its cardinality.
For a closed set $\mathcal{S}$, we use $\iota_{\cal S}(\mathrm{x})$ to denote its
 indicator function. Given a matrix $A \in \mathbb{R}^{n \times n}$ or a vector $\h{x} \in \mathbb{R}^{n}$ and an index set $\Lambda \subseteq[n]$,  $A_{\Lambda,\Lambda}$, ${\h x}|_{\Lambda}$ to denote
$A(i,j)_{i,j\in\Lambda}$ and ${\h {x}}(i)_{i \in \Lambda}$, respectively.
 $A\succ{\bf 0}$ means positive definite, and $\lambda_{\min}(A)$ denotes the minimum eigenvalue of $A$.
 ``{\text{\rm dist}}" represents the distance function.
A ${\cal C}^1$ function $f$ means its gradient Lipschitz continuous. Analogously, we can define ${\cal C}^p$ function ($p\ge2$).

A set ${\cal M}$ is a manifold around a point ${\h x}$ if ${\h x}\in{\cal M}$ and
there is an open set $V$ containing ${\h x}$ such that
${\cal M}\cap V=\{ {\h x}\in V | \Phi({\h x})=0\}$ where
the smooth function $\Phi$ has surjective derivative throughout $V$.

 The polar cone of a set $C$ is defined by
$C^o=\{ {\h y}\in {\mathbb R}^n : \langle{\h y},{\h x} \rangle \le 0, \;\forall\; {\h x}\in C\}.$
An extended-real-valued function $f:{\mathbb R}^n\rightarrow (-\infty,+\infty]$ is said to be
 proper if ${\text{dom}} f=\{{\h x}\in {\mathbb R}^n| f(\h x)<\infty \}$ is nonempty.   We denote the extended reals by ${\overline{\mathbb{R}}}=[-\infty,+\infty]$.

 We review some definitions from \cite{RockWetsVA}. Consider a function $h: {\mathbb R}^n\rightarrow {\overline{\mathbb{R}}}$ finite at a point ${\h x}\in{\mathbb R}^n$,
 the subderivative $dh({\h x})(\cdot): {\mathbb R}^n\rightarrow {\overline{\mathbb{R}}}$ is defined by
 \begin{eqnarray*} dh({\h x})({\bar{\h w}}) = \mathop{\lim_{\tau\downarrow 0}\inf_{{\h w}\to {\bar{\h w}}}}\frac{h({\h x}+\tau {\h w})-h({\h x})}{
 \tau}, \end{eqnarray*}
and the regular subdifferential ${\hat {\partial}} h({\hat{\h x}})$
 and the limiting subdifferential $\partial h({\hat{\h x}})$   are defined as
\begin{eqnarray*}
\begin{array}{l}
		{\hat\partial} h({\hat{\h x}})=\left\{{\h v}\Big |{\mathop{\lim}\limits_{\h x\to{\hat{\h x}}}}\inf\limits_{{\h x}\neq {\hat{\h x}}} {\displaystyle\frac{h(\h x)-h(\hat{\h x})-\langle {\h v},{\h x}-{\hat{\h x}}\rangle}{\|{\h x}-\hat{\h x}\|_2}}\ge 0\right\},\\[0.4cm]
		\partial h({\hat{\h x}})=\left\{{\h v}\Big | \exists {\h x}^r\rightarrow {\hat{\h x}},h({\h x}^r)\rightarrow h(\hat{\h x}),
 {\h v}^r\in{\hat\partial} h({\h x}^k),\! {\h v}^k\rightarrow{\h v}\!\! \right\},
		\end{array}
\end{eqnarray*}
respectively.
The  horizon subdifferential is defined by
\begin{eqnarray*} \partial^\infty h({\h x})\!\!=\!\!\left\{\lim\limits_{r} \lambda_r {\h v}^r\Big| {\h v}^r\in \partial h({\h x}^r),{\h x}^r\to {\h x},
h({\h x}^r)\to h({\h x}),\lambda_r \downarrow 0\right\}.\end{eqnarray*}
Suppose that $h({\hat{\h x}})$ is finite and $\partial h({\hat{\h x}})\neq\emptyset$,
$h$ is {\it regular} at ${\hat{\h x}}$ if and only if $h$ is locally lower semicontinuous at ${\hat{\h x}}$
with $\partial h({\hat{\h x}}) ={\hat\partial} h({\hat{\h x}})$ and
$\partial^\infty h({\hat{\h x}}) ={\hat \partial}h({\hat {\h x}})^\infty$ \cite[Corollary 8.11]{RockWetsVA}, where ${\hat \partial}h({\hat {\h x}})^\infty$  is recession cone (in the sense of convex analysis).

Let $f$ be a proper lower semicontinuous function,  its proximal mapping defines as $$\prox_f({\h v}) = \arg\min_{\h x} \Big\{ f({\h x}) + \frac{1}{2}\|{\h x}-{\h v}\|_2^2\Big\}.$$
  By defining ${\frac{\|{\bf 0}\|_1}{\|\bf 0\|_2}}=1$ \footnote{For any vector ${\h x}\in{\mathbb R}^n$, $\frac{\|{\h x}\|_1}{\|\h x\|_2}\ge 1$ due
  to Cauchy-Schwarz inequality. Another benefit of defining $\frac{\|{\bf 0}\|_1}{\|{\bf 0}\|_2}=1$ is possible to exclude
  the stationary being zero vector. We refer the reader to \cite{Tao20,TaoZhang23} for more discussions.}, the objective function of the following
  proximal operator
  is lower semicontinuous,
\begin{eqnarray}\label{close} \prox^{\rho}_{[(\|{\cdot}\|_1+\iota_{\cal X}(\cdot))/\|{\cdot}\|_2]}(\h q)\!\!=\!\!\mathop{{\arg\min}}_{\h x\in{\cal X}}\!\Big(\frac{\|\h x\|_1}{\|\h x\|_2}\!+\!\frac{\rho}{2}\|\h x-\h q\|_2^2\Big). \nn\\
\end{eqnarray}
Thus, the proximal operator of $ \prox^{\rho}_{[(\|{\cdot}\|_1+\iota_{\cal X}(\cdot))/\|{\cdot}\|_2]}$ is well-defined
 due to \cite[Definition 1.22]{RockWetsVA}.
One of {\it the proximal points}
 for ${\cal X}={\mathbb R}^n$ and ${\mathbb R}^n_+$ has been characterized in
{\it a closed-form}  in \cite[Theorem 3.3]{Tao20},
\cite[Theorem 6]{TaoZhang23},
respectively.

\noindent Next, we review several concepts from variational analysis.
\begin{Def}\label{defprox} ({\bf Prox-regularity}) \cite[Definition 2.1]{PR96}
A function $f$ is prox-regular at a point ${\bar{\h x}}$ for a subgradient ${\bar {\h v}}\in \partial f({\bar{\h x}})$ if
$f$ is finite at ${\bar{\h x}}$, locally lower semi-continuous around ${\bar{\h x}}$, and
there exists $\rho>0$ such that
$$ f({\h x}')\ge f({\h x})+ \langle{\h v},{\h x}'-{\h x} \rangle-\frac{\rho}{2}\|{\h x}'-{\h x}\|_2^2$$
whenever ${\h x}$ and ${\h x}'$ are near ${\bar{\h x}}$ with $f({\h x})$ near $f({\bar{\h x}})$ and ${\h v}\in \partial f({\h x})$ is
near ${\bar{\h v}}$.
Furthermore, $f$ is prox-regular at ${\bar{\h x}}$ if it is prox-regular at ${\bar{\h x}}$ for every  ${\bar{\h v}}\in\partial f({\bar{\h x}})$.
\end{Def}
The terminology ``partly smooth" as defined in \cite[Definition 2.7]{Lewis02} differs from that presented in \cite[Definition 2.3]{LW04}. Notably, the difference arises from the level of smoothness of the concerned function $f$. While the characterization of ``partly smooth" in \cite[Definition 2.3]{LW04} necessitates the function $f$ to exhibit ${\cal C}^p$ smooth ($p\ge 2$), the criteria outlined in \cite[Definition 2.7]{Lewis02} imposes a more relaxed condition of the function $f$ being ${\cal C}^1$. We  adopt  the latter criterion.
\begin{Def}\label{def2.3} ({\bf Partly Smooth}) \cite[Definition 2.7]{Lewis02} Suppose that the set ${\cal M}\subset{\mathbb R}^n$ contains the point $\h x$. The function $f: {\mathbb R}^n\rightarrow{\overline {\mathbb R}}$
is partly smooth at $\h x$ relative to ${\cal M}$ if ${\cal M}$ is a manifold around $\h x$ and the following four properties hold:
\begin{itemize}
\item[(i)]({\bf Restricted Smoothness}) the restriction $f|_{\cal M}$ is smooth around $\h x$;
\item[(ii)] ({\bf Regularity}) at every point close to $\h x$ in ${\cal M}$, the function $f$ is regular and has a subgradient;
\item[(iii)] ({\bf Normal Sharpness}) $df({\h x})(-{\h w})>-df({\h x})({\h w})$ for all nonzero directions $\h w$ in $N_{\cal M}({\h x})$;
\item[(iv)] ({\bf Subgradient Continuity}) the subdifferential map $\partial f$ is continuous at $\h x$ relative to ${\cal M}$.
\end{itemize}
\end{Def}
We say $f$ is partly smooth relative to a set ${\cal M}$ if ${\cal M}$ is a manifold and $f$ is partly smooth
at each point in ${\cal M}$ relative to ${\cal M}$.
Property (i) ensures that $f$ is continuous relative to ${\cal M}$, therefore  the subdifferential mapping is always
outer semicontinuous relative to ${\cal M}$ \cite[Proposition 8.7]{RockWetsVA}. Thus, property (iv) can be replaced
with ``the subdifferential $\partial f({\h x})$ is {\it inner semicontinuous} at ${\h x}$ relative to ${\cal M}$".
It means that for any sequence  ${\h x}^r$ in ${\cal M}$ approaching ${\h x}$ and
any subgradient ${\h y}\in\partial f({\h x})$, there exists subgradients ${\h y}^r\in \partial f({\h x}^r)$ approaching ${\h y}$.

Below, we review the concepts of semismooth and strongly semismooth \cite{doi:10.1137/0315061} and  LC$^1$ function \cite{FACCHINEI1995131}.
\begin{Def} $F: {\mathbb R}^n\rightarrow {\mathbb R}$ is directionally differentiable at ${\h x}$.
We say that  $F$ is
{\it semismooth} at ${\h x}$ if  for  any $V\in\partial F({\h x}+{\h h})$,
\begin{eqnarray*}\label{sem} F({\h x}+{\h h})-F({\h x})-V{\h h}=o(\|{\h h}\|_2).\end{eqnarray*}
$F$ is said to be {\it strongly semismooth} at ${\h x}$ if  for any  $V\in\partial F({\h x}+{\h h})$,
\begin{eqnarray*} F({\h x}+{\h h})-F({\h x})-V{\h h} =O({\|{\h h}\|_2^2}).\end{eqnarray*}
\end{Def}
\begin{Def}
A function $f:{\mathbb R}^n \rightarrow{\mathbb R}$ is said
to be an {\rm LC}$^1$ function on an open set $\Omega$ if
$f$ is continuously differentiable on $\Omega$, and
$\nabla f$ is locally lipschitz on $\Omega$.
\end{Def}
For a ${\rm LC}^1$ function $f$, its gradient is locally Lipschitz on $\Omega$; the gradient is differentiable almost everywhere in $\Omega$, allowing for the definition of its generalized Jacobian in Clark's sense, also referred to as the generalized Hessian of $f$ \cite{Clark90}. We define the generalized Hessian of $f$ at $\mathbf{x}$ to be the set $\partial^2 f(\mathbf{x})$ of $n\times n$ matrices defined by:
\begin{eqnarray*}&\partial^2 f(\h x):={\rm co}\{ H\in\mathbb R^{n\times n}: \exists \;{\h x}^k\to {\h x} \;\mbox{with}\nn\\
&\nabla f\;\mbox{differentiable at}\;\h x^k\;\mbox{and}\; \nabla^2 f(\h x^k)\to H\}.\end{eqnarray*}
Note that $\partial^2 f$ is a nonempty, convex, and compact set of symmetric matrices.

We proceed to verify the following lemma, a pivotal step in establishing the prox-regularity of the fraction function.
 \begin{lemma}
 \label{proxreg}
 Let $f,\;g:{\mathbb R}^n \rightarrow (-\infty,+\infty]$ be proper lower semicontinuous functions,
 and  finite at
${\bar {\h x}}\in {\mathbb R}^n$. Suppose that $f$ and $g$ are locally Lipschitz continuous around ${\bar{\h x}}$ and
$g({\bar{\h x}})> 0$ and $f$ is prox-regular at the point ${\bar{\h x}}$.
$g$ is strictly differentiable at ${\bar{\h x}}$,
and $\nabla g$ is locally Lipschitz continuous around ${\bar{\h x}}$. Then, the function $\displaystyle{\frac{f}{g}}$ is prox-regular at the point
${\bar{\h x}}$.
 \end{lemma}
\begin{proof}See appendix \ref{AppA}. \end{proof}

\begin{remark}
The significance of this lemma lies in its pivotal role in establishing the prox-regularity property of $L_1$ over $L_2$ (with
 or without nonnegative constraint) regularization function,
 which  is a crucial step to establish the finite identification of  ADMM$_p$ and ADMM$_p^+$.
\end{remark}
\section{Partly smooth}\label{wellDef}
In this section, we illustrate the partly smoothness of the model (\ref{L1o2uncon})  whenever ${\cal X}={{\mathbb R}^n_+}/{\mathbb R}^n$.
To carry out a unified analysis, we denote
\begin{eqnarray} \label{hfun} h(\h x)=\frac{\|\h x\|_1}{\|\h x\|_2}+\iota_{\cal X}(\h x).\end{eqnarray}
For any vector ${\h x}_0(\neq{\bf 0})$, we define the manifold
\begin{eqnarray*} {\cal M}_{{\h x}_0}=\{ {\h x}\in{\mathbb R}^n \;|\; {\text{supp}}({\h x})={\text{supp}}({\h x}_0)\}.\end{eqnarray*}
We first characterize the normal cone (normal space) to ${\cal M}_{{\h x}_0}$.
$$N_{{\cal M}_{{\h x}_0}}(\h x)=\left\{{\h w}\in {\mathbb R}^n\;|\; {\h w}_j=0,\; j\in{\text{supp}}({\h x}_0)\right\}.$$
\begin{proposition} \label{PhiPS} Suppose that $\Phi({\h x})$ is smooth, the objective function
$F_u({\h x})+\iota_{\cal X}(\h x)$ ($F_u({\h x})$ defined in (\ref{L1o2uncon})) is partly smooth at ${\h x}_0\;(\neq{\bf 0})$ relative to ${\cal M}_{{\h x}_0}$.\end{proposition}
\begin{proof} See appendix \ref{AppB}. \end{proof}

\begin{remark}
By taking $\Phi({\h x})=\frac{1}{2}\|A{\h x}-{\h b}\|_2^2$ in (\ref{L1o2uncon}), the resulting model is partly smooth (i.e., the model (1.2) \cite{Tao20}
and the model (2) \cite{TaoZhang23}).\end{remark}

\begin{remark} If $\Phi({\h x})$ is smooth except some points (the set of these irregular points denoted by ${\cal J}$).
Then, the objective function $F_u$ in (\ref{L1o2uncon}) is partly smooth at ${\h x}_0$ (${\h x}_0\neq{\bf 0}$ and ${\h x}_0\not\in{\cal J}$)
relative to ${\cal M}_{{\h x}_0}$. For example, $\Phi({\h x})=\|A{\h x}-{\h b}\|_2$ and
${\cal J}=\{ {\h x}|A{\h x}-{\h b}=0\}$.
The resulting $F_u$ in (\ref{L1o2uncon}) is partly smooth at ${\h x}_0$ (${\h x}_0\neq{\bf 0}$ and ${\h x}_0\not\in{\cal J}$) relative
to ${\cal M}_{{\h x}_0}$.
 \end{remark}

\section{Activity identification}\label{ADMMS}

In \cite{Tao20, TaoZhang23}, a closed-form solution for the proximal operator of $(L_1/L_2)$ and $(L_1/L_2)^+$ was derived in \cite[Theorem 3.3]{Tao20} and \cite[Theorem 6]{TaoZhang23}, respectively. Additionally, practical solvers for finding global optimizers of (\ref{close}) for ${\cal X}={\mathbb R}^n$ and ${\cal X}={\mathbb R}^n_+$ were developed, respectively; see, e.g.,  \cite[Algorithm 3.1]{Tao20} and \cite[Algorithm 1]{TaoZhang23}.
These practical solvers have been utilized in the context of ADMM$_p$ \cite[Algorithm 4.1]{Tao20} and ADMM$_p^+$ \cite[Algorithm 2]{TaoZhang23}, demonstrating excellent numerical performance for solving (\ref{L1o2uncon}) with $\Phi(\h x)=\frac{1}{2}\|A{\h x}-{\h b}\|_2^2$.
We unify ADMM$_p$ and ADMM$_p^+$ with general $\Phi$ as follows:

\begin{subequations} \label{ADMMschI}
	\begin{numcases}{\hbox{\quad}}
	\label{xsub-ADMMI}\h x^{k+1}\in \prox^{\beta_{\gamma}}_{[(\|{\cdot}\|_1+\iota_{\cal X}(\cdot))/\|{\cdot}\|_2]}\left({\h y}^k-\frac{1}{\beta}{\h z}^k\right)\\[0.1cm]
	\label{ysub-ADMMI}\h y^{k+1}=\arg\min_{\h y} \Phi(\h y)+\frac{\beta}{2}\left\|{\h y}-{\h x}^{k+1}-\frac{{\h z}^{k}}{\beta}\right\|_2^2 \\[0.1cm]
	\label{zsub-ADMMI}\h z^{k+1}=\h z^k+\beta(\h x^{k+1}-\h y^{k+1}),
	\end{numcases}
\end{subequations}
where $\beta_{\gamma}=\beta/\gamma$.
Our previous studies established the global convergence of (\ref{ADMMschI}) when  ${\cal X}={\mathbb R}^n$ or ${\cal X}={\mathbb R}_+^n$, respectively in \cite{Tao20} and \cite{TaoZhang23}. In this paper, we aim to show the ability of ADMM$_p$ and ADMM$_p^+$ to identify the active manifold.
 In doing so, we first present a unified exposition of its convergence properties for the general $\Phi$.
 The proof is similar to \cite[Theorem 5.4]{Tao20} and \cite[Theorem 7]{TaoZhang23}.
 We
omit the proof here for brevity.

\begin{theorem}\label{theo2}
		Let $\{{\h w}^k:=({\h x}^k,{\h y}^k,{\h z}^k)\}$ be the sequence generated by   (\ref{ADMMschI}).	If $\beta>2L$ ($L$ denotes the Lipschitz constant of $\nabla \Phi$), we have
			 $\lim_{k\to\infty}\|{\h y}^k - \h y^{k+1}\|_2=0$.
	\end{theorem}
The global convergence of (\ref{ADMMschI}) is summarized below  \cite{Tao20,TaoZhang23}.
\begin{theorem}\label{thmglobal}
	Let  $\{\h w^k\}$ be the sequence generated by (\ref{ADMMschI}).
If  ${A^\top {\h b}}\not\in {\cal X}^o$ and $\beta>2L$, and $\{\h x^k\}$ is bounded, then
\begin{itemize}
\item[(i)]for ${\cal X}={\mathbb R}^n$, $\{{\h x}^k\}$ converges to  a stationary point ${\h x}^\infty$, i.e.,
\begin{eqnarray*}
0\in\gamma\left(\frac{\sign(\h x^{\infty})}{r}-\frac{a}{r^3}{\h x}^{\infty}\right)+\nabla\Phi({\h x}^\infty),\end{eqnarray*}
where $a=\|{\h x}^\infty\|_1$ and $r=\|{\h x}^\infty\|_2$;\\[-0.2cm]
\item[(ii)] for ${\cal X}={\mathbb R}^n_+$, $\{\h x^k\}$ converges to a d-stationary point \cite{9186389,PMA} of (\ref{L1o2uncon}), i.e.,
\begin{eqnarray*}\left\langle{\h x}-{\h x}^\infty, \gamma\left(\frac{{\h e}}{r}-\frac{a}{r^3}{\h x}^\infty\right) +\nabla \Phi({\h x}^{\infty})\right\rangle\ge 0,\nn\\
\quad\quad\quad\quad\quad\quad\quad\quad\quad\quad\;\forall\; {\h x}\in{\cal X}.\end{eqnarray*}
\end{itemize}
\end{theorem}
\subsection{Finite Identification of ADMM$_p$ and ADMM$_p^+$}
We elaborate on  the primary findings of (\ref{ADMMschI}) related to the identification of the true support set within a finite steps.
 First, we relax the requirement of ``the function $f$ be ${\cal C}^p$-partly smooth ($p\ge 2$)" as stated in \cite[Theorem 5.3]{LW04} to the condition of ``the function $f$ be partly smooth, i.e., $f$ be ${\cal C}^1$" (see Definition \ref{def2.3}).
This relaxation allows us to establish the finite identification property of ADMM$_p$ (or ADMM$_p^+$) under more relaxed  conditions.

 \begin{theorem}\label{identify}
Let the function $f$ be partly smooth at the point ${\bar{\h x}}$ relative to the manifold ${\cal M}$,
and prox-regular there, with ${\bf 0}\in{\text{\rm ri}}\partial f({\bar{\h x}})$.
Suppose ${\h x}_k\to {\bar{\h x}}$ and $f({\h x}_k)\to f({\bar{\h x}})$. Then,
${\h x}_k\in{\cal M}\; \mbox{for all large}\; k$
if and only if
${\text{\rm dist}}({\bf 0},\partial f({\h x}_k))\to 0.$
\end{theorem}
\begin{proof} See the proof to Theorem S.7 in supplementary materials. \end{proof}
Based on this, we show the finite identification property of (\ref{ADMMschI}).
\begin{theorem}
\label{thmidentify} 
	Let  $\{{\h w}^k\}$ be the sequence generated by  (\ref{ADMMschI}).
Suppose that ${A^\top {\h b}}\not\in {\cal X}^o$, $\beta>2L$ and $\{{\h x}^k\}$ is bounded.
Then, the sequence $\{{\h w}^k\}$ converges to ${\h w}^\infty$.
Suppose that $\Phi({\h x})$ is  ${\cal C}^1$ and prox-regular at ${\h x}^\infty$.
If $0\in{\text{\rm ri}}(\partial (F_u(\cdot)+ \iota_{\cal X}(\cdot))({\h x}^\infty))$,
$${\h x}_k\in{\cal M}_{{\h x}^\infty},\;\;\mbox{for all large}\;k.$$
\end{theorem}
\begin{proof}See appendix \ref{AppC}. \end{proof}

To this end, we further exploit the condition of $$0\in{\text{\rm ri}}(\partial (F_u(\cdot)+ \iota_{\cal X}(\cdot))({\h x}^\infty)).$$
We can characterize it by dividing it into two cases:

(a) When ${\cal X}={\mathbb R}^n$,
\begin{eqnarray*}
&&0\in{\text{\rm ri}}(\partial (F_u(\cdot)+ \iota_{\cal X}(\cdot))({\h x}^\infty) \nn\\[0.1cm]
 &&\Leftrightarrow  \gamma\left(\frac{{\text{sign}}(({\h x}^\infty)_{\Lambda})}{r}-\frac{a}{r^3}({{\h x}^\infty})_{\Lambda}\right)+({\h q}^\infty)_{\Lambda}=0,  \nn\\
 &&\mbox{and}\;\|{({\h q}^\infty)}_{{\Lambda}^c}\|_{\infty}<\gamma/r,
\end{eqnarray*}
where ${\h q}^\infty=\nabla \Phi({\h x}^\infty)$, $a=\|{\h x}^\infty\|_1$, and $r=\|{\h x}^\infty\|_2$.

(b) When ${\cal X}={\mathbb R}_+^n$,
\begin{eqnarray*}
&&0\in{\text{\rm ri}}(\partial (F_u(\cdot)+ \iota_{\cal X}(\cdot))({\h x}^\infty)\nn\\[0.1cm]
&&\Leftrightarrow  \gamma\left(\frac{{\bf e}}{r}-\frac{a}{r^3}({{\h x}^\infty})_{\Lambda}\right)+({\h q}^\infty)_{\Lambda}=0,\nn\\
 &&\mbox{and}\;{({\h q}^\infty)}_{{\Lambda}^c}+\frac{\gamma}{r}>0,
\end{eqnarray*}
where ${\h q}^\infty$, $a$, and $r$ are defined the same as in Case (a).
\begin{remark}
Thus, we establish the identification property of both ADMM$_p$ and ADMM$_p^+$. This property implies the existence of an index number, denoted as $K$, such that the sequence $\{{\h x}^k\}$ generated by either ADMM$_p$ or ADMM$_p^+$ will eventually enter the active manifold. This serves as the theoretical basis for designing an acceleration framework provided in Section \ref{HAF}.

\end{remark}

\subsection{Heuristic Acceleration Framework}\label{HAF}
 As proven in Theorem \ref{thmidentify}, both ADMM$_p$ and ADMM$_p^+$ can successfully identify the active manifold within a finite number of iterations. To leverage this result, we develop a heuristic acceleration framework for the active manifold (abbr. HAFAM).
This heuristic framework comprises two phases:
(I) Utilize ADMM$_p$ or ADMM$_p^+$ to identify the active manifold;
(II) Employ a globalized semismooth Newton method (abbr.  {\bf SSNewton}) to solve the reduced low-dimensional minimization problem.

The primary advantage of this approach lies in the dimension of the identified manifold, which is typically much smaller than the original. For implementation details of HAFAM, refer to Algorithm \ref{alg:ACT}. Additionally, for Phase II in HAFAM, consult Algorithm \ref{ALG_NCG}.

To determine the transition from Phase I to Phase II, we utilize a heuristic strategy, ensuring the support set remains unchanged for consecutive $T$ iterations. Mathematically,
\begin{eqnarray}\label{adjLam}\Lambda_k = \Lambda_{k+1} = \cdots = \Lambda_{k+T}, \quad (\Lambda_k = {\text{supp}}({\h x}^{(k)})). \end{eqnarray}
To elucidate Phase II, let ${\h x}^{I}$ be generated by ADMM$_p$ or ADMM$_p^+$ (\ref{ADMMschI}) such that (\ref{adjLam}) is satisfied. Denote ${\widehat\Lambda}={\text{supp}}({\widehat{\h x}})$, and define
\begin{eqnarray}\label{Mstar} \widehat{\cal M} = \{ {\h x}\in{\mathbb R}^n \;|\; {\text{supp}}({\h x})={\widehat\Lambda}\}. \end{eqnarray}
Below, we focus on ${\cal X}={\mathbb R}^n$ for brevity.

Without loss of generality, we assume $$({\h x})^\top=\left(({\h x}_{\widehat{\Lambda}})^\top,({\bf 0})^\top\right)$$ for any ${\h x}\in\widehat{\cal M}$.
Therefore, we can characterize
the objective function of (\ref{L1o2uncon}) as $\varphi: {\mathbb R}^s\rightarrow
{\mathbb R}$ (let ${\h u}={\h x}_{\widehat\Lambda}$, $s={\text{dim}}({\h u})$),
\begin{eqnarray}\label{varphi} \varphi({\h u}) =  F_u\left(\begin{array}{c}{\h u}\\{\bf 0}\end{array}\right). \end{eqnarray}
Thus, the original problem (\ref{L1o2uncon}) is transformed into an equivalent problem within a lower-dimensional space:
\begin{eqnarray}\label{reducedim}\min_{{\h u}\in{\mathbb R}^s}  \varphi({\h u}).\end{eqnarray}
 Phase II aims to solve the problem (\ref{reducedim}).

\begin{algorithm}[tt]
	\caption{Heuristic Acceleration Framework on Active Manifold (HAFAM)}
	\label{alg:ACT}
	\begin{algorithmic}[1]
		\Require{$A\in{\mathbb R}^{m\times n}$, $\h b\in{\mathbb R}^m$, ${\h x}_0\in\mathbb R^n$, $\beta>0,\;\tau\ge 0,\;\gamma>0, \;T\in{\mathbb N}$. }
        \State {\bf Phase I: Identify the active manifold by ADMM$_p$ (or ADMM$_p^+$).}
        \State Initialize: $\h y^0=\h z^0={\h x}_0$.
		\While{(\ref{adjLam}) not satisfied.}
		\State Solve $\h x^{k+1}$ via (\ref{xsub-ADMMI}).
		\State Compute $\h y^{k+1}$ via  (\ref{ysub-ADMMI}).
		\State Update $\h z^{k+1}$ via (\ref{zsub-ADMMI}).
		\EndWhile
        \State Output ${\h x}^{I}$.
        \State Compute {${\widehat{\h x}} ={\text{HARD}}({\h x}^{I},\tau)$}.
        \State Let $\widehat{\cal M}$ be defined in (\ref{Mstar}).
        \State {\bf Phase II:  Globalized semismooth Newton method.}
        \State Use Algorithm \ref{ALG_NCG} to solve (\ref{reducedim}), and
        \State  let ${\h u}$={\bf SSNewton}$({\widehat{\h x}},\widehat{\cal M})$.
         \State Output $\overline{\h x}$ such that ${\overline{\h x}}|_{\widehat{\cal M}}={\h u}$ and $\overline{\h x}|_{{\widehat{\cal M}}^c}={\bf 0}$.
	\end{algorithmic}
\end{algorithm}

If $\Phi(\cdot)$ is LC$^1$-function,  we calculate
\begin{eqnarray}\label{varphigrad} \nabla \varphi({\h u})=\gamma\left(\frac{\sign{{\h u}}}{\|{\h u}\|_2}-\frac{\|{\h u}\|_1}{\|{\h u}\|_2^3}{\h u}\right)+\nabla_{\h u} \Phi\left(\begin{array}{c}{\h u}\\{\bf 0}\end{array}\right),\end{eqnarray}
\noindent and the generalized Hessian of $\varphi$
\begin{eqnarray}\label{LocCond}&&\partial^2 \varphi({\h u})\! =\!\{{\cal V}\,|\,{\cal V}=V_{\widehat\Lambda,\widehat\Lambda}\nn\\
&&-\gamma\frac{1}{r^2}\left( \frac{{\h u}({ \sign(\h u)})^\top +{ \sign(\h u)}{\h u}^\top}{r}\right)+\gamma\frac{3a{\h u}{\h u}^\top}{r^5}-\gamma\frac{a}{r^3}\},\nn\\
  \end{eqnarray}
  where $V\in \partial^2 \Phi({\h x})$,
  and $a=\|{\h u}\|_1$ and $r=\|{\h u}\|_2$.
  By defining $$Q=\frac{\gamma}{r^3}\left[\left({\h u}({ \sign(\h u)})^\top +{ \sign(\h u)}{\h u}^\top\right)-\frac{3a{\h u}{\h u}^\top}{r^2}+{a}I\right],$$
  we can show that all the eigenvalues of $Q$ are as follows:
  \begin{eqnarray*}\lambda_{1}=\frac{\gamma(a- \sqrt{4s r^2 -3a^2})}{2r^3},\;\lambda_2=\frac{\gamma(a+ \sqrt{4s r^2 -3a^2}) }{2r^3},\end{eqnarray*}
 accompanied  with $0$ eigenvalue in the multiplicity of $s-2$.
Obviously, $\lambda_1<0$ and $\lambda_2>0$.

With these preparatory results, the globally semismooth Newton method can be applied to solve  (\ref{reducedim}), and
more details can be found in Algorithm \ref{ALG_NCG}.

\begin{algorithm}[t]
	\caption{\;\;{\bf SSNewton}$({\h x}^0,{\cal M})$}\label{ALG_NCG}
	\begin{algorithmic}[1]
			\State Initialization: $\h u^{0}={\h x}^0|_{\cal M}\in \mathbb{R}^{m},\, \mu\in (0,\, 1/2),\, {\eta}, \nu\in (0,\, 1)$ and $\delta\in (0,\, 1)$. Define $\varphi$ and $\nabla \varphi$ by (\ref{varphi}), (\ref{varphigrad}), respectively.
			\While{``not converge"}
            \State{Let $V_j\in\partial^2\varphi({\h u}^j)$. Find  an approximate solution ${\h d}^j$} to
             \begin{eqnarray}\label{newton} V_{j}\h{d}^j=-\nabla \varphi (\h{u}^{j}), \end{eqnarray}
            \State{such that}
            \begin{eqnarray}\label{vj}\|\nabla \varphi({\h u}^j)+V_j {\h d}^j\|_2\le \eta_j\|\nabla\varphi({\h u}^j)\|_2,\end{eqnarray}
            \State{where $\eta_j=\min\{\eta,\|\nabla\varphi({\h u}^j)\|_2\}$.}
            \If {(\ref{vj}) is unachievable or $\langle\nabla \varphi({\h x}^j),{\h d}^j \rangle\!>\!-\nu_j\|{\h d}^j\|_2^2$}
            \State{$\quad$${\h d}^j=-B_j^{-1}\nabla \varphi({\h d}^j)$.}
             \EndIf
          \State{where $\nu_j=\min\{\nu, \|\varphi({\h u}^j)\|_2\}$ and $B_{j}\succ{\bf 0}$.}
                  \For{$m=0,\, 1,\, 2,\, \cdots$}
			\State {Set $\alpha_{j}=\delta^{m}$,
                    \State If $\alpha_{j}$ satisfies
					\State
						$\varphi (\h{u}^{j}+\alpha_{j}\h{d}^{j}) \leq\, \varphi (\h{u}^{j}) +\mu \alpha_{j}\langle \nabla \varphi (\h{u}^{j}),\, \h{d}^{j}\rangle.$}
\State Set $\h{u}^{j+1}=\h{u}^{j}+\alpha_{j}\h{d}^{j}.$
               \EndFor
			\EndWhile
\State Output ${\h u}$.
	\end{algorithmic}
\end{algorithm}

Except the two phases in Algorithm \ref{alg:ACT}, there is  an additional hard-shrinkage step, and the operator ${\text{HARD}}$ is defined as:
\begin{eqnarray*}
{\text{\rm HARD}}({\h x},\tau) = \max(|{\h x}| - \tau, 0) .* {\h x},
\end{eqnarray*}
where $.*$ denotes componentwise multiplication.
 This step is to address the recovery under noisy data.
More specifically, we set $ \tau = 0 $ when the data is noiseless and $ \tau > 0 $ when the observation is noisy.

\begin{remark} In Algorithm \ref{ALG_NCG}, we practically use the conjugate gradient (CG) method to solve the following
perturbed Newton equation:
\begin{eqnarray*}(V_{j}+\varepsilon_k I)\h{d}^j+\nabla \varphi (\h{u}^{j})=0, \end{eqnarray*}
where $\varepsilon_k =\|\varphi (\h{u}^{j})\|_2$.
\end{remark}

\begin{remark} In Algorithm \ref{ALG_NCG}, the matrix $B_j$ is any real symmetric positive definite matrix for all $j$.
\end{remark}
\subsection{High-Order Convergence Guarantees}
We establish the proposed approach's eventual superlinear/quadratic convergence rate under the assumption that the objective function possesses a smooth representation that is semismooth/strongly semismooth, along with some additional conditions.  The analysis of the following theorem is inspired by \cite[Theorem 5.3]{QiSun} and \cite{ZhaoSunToh}. For the sake of completeness, we have included the proof here.
\begin{theorem}\label{sublim}
Suppose the assumptions in Theorem \ref{thmidentify} are fulfilled and $\Phi(\h x)$ is $LC^{1}$ and bounded below.
Both $\{\|B_j\|_2\}$ and $\{\|B_j^{-1}\|_2\}$ are uniformly bounded.
Let $\{{\h u}^j\}$ be generated by Algorithm \ref{ALG_NCG} with ${\cal M}={\cal M}_{{\h x}^\infty}$.
\begin{itemize}
 \item[(i)] Any accumulation point ${\hat{\h u}}$ of
$\{{\h u}^j\}$
 is a critical point of (\ref{reducedim}). 
\item[(ii)] Furthermore, assume that $\varphi$ is level bounded on ${\cal M}_{{\h x}^\infty}$.
If $\nabla \varphi$ is semismooth at ${\hat{\h u}}$ and for all ${\hat V}\in\partial^2 \varphi({\hat{\h u}})$ are positive definite, then
the whole sequence $\{{\h u}^j\}$ converges to ${\hat{\h u}}$ superlinearly; if $\nabla \varphi$ is strongly semismooth at ${\hat{\h x}}$, then
the sequence $\{{\h  u}^j\}$ converges to ${\hat{\h u}}$ quadratically.
\item[(iii)]
By defining ${\widehat{\h x}}|_{\widehat\Lambda} = {\hat{\h u}}$ and ${\widehat{\h x}}|_{{\widehat\Lambda}^c} = {\bf 0}$,
${\widehat {\h x}}$ recovers a critical point of (\ref{L1o2uncon}).
\end{itemize}
\end{theorem}
\begin{proof} See appendix \ref{AppD}. \end{proof}
\begin{remark} Suppose $V_{\widehat\Lambda,\widehat\Lambda}\succ{\bf 0}$ defined in (\ref{LocCond}) with ${\h u}={\hat{\h u}}$.
There exists a scalar $\bar\gamma>0$ (e.g., ${\bar\gamma}=\frac{2\|{\hat{\h u}}\|_2^2\lambda_{\min}(V_{\widehat\Lambda,\widehat\Lambda})}{3\sqrt{\|{\hat{\h u}}\|_0}}$) such that when $0<\gamma<{\bar\gamma}$,  then
   all the ${\hat V}\in\partial^2 \varphi({\hat{\h u}})$ are positive definite.
\end{remark}

\section{Numerical results}\label{NumRes}
We present numerical results to illustrate the efficiency of the proposed algorithmic framework
(i.e., Algorithm \ref{alg:ACT}) and its ability to achieve high-order convergence rates and high-accuracy solutions.
The efficiency of ADMM$_p$/ADMM$_p^+$ for arbitrary/nonnegative sparse recovery has been
well demonstrated in \cite{Tao20,TaoZhang23} in comparison with the existing state-of-the-art methods.
Below, we focus on the comparisons between Algorithm \ref{alg:ACT} and ADMM$_p$.
All these algorithms are implemented on MATLAB R2016a, and performed on a desktop with Windows 10 and an Intel Core i7-7600U CPU processor (2.80GH) with 16GB memory.
The stopping criterion of ADMM$_p$ we adopt relative error (RelErr) less  than a threshold:
\begin{eqnarray}\label{StopC}
&&{\text{RelErr}}=\frac{||\h x^{k-1}-\h x^{k}||_2}{\max\{10^{-16},||\h x^{k}||_2,||\h x^{k-1}||_2\}} < 10^{-8} \nn\\
&&\mbox{ or}\; k_{\max }>{\rm IMax}.
 \end{eqnarray}

Let ${\rm IMax}=2000$.
Set $T=5,10,20,30$ in Algorithm \ref{alg:ACT}, denoted as HAFAM$_T$.
The setting of parameter $\beta$ in HAFAM$_T$ is the same as ADMM$_p$.
For {\bf SSNewton}$({\widehat{\h x}},\widehat{\cal M})$, we take
${\eta}=1e-3$, $\nu=1e-8$, $\mu=1e-8$, $\delta=0.95$ and $B_j=0.1I$
for all $j$. We used
$$\|\nabla \varphi({\h u}^j)\|_2\le 1e-11\;\;\mbox{or}\;\;j> {\text{SSNMax}}$$
as the termination criterion
and  ${\text{SSNMax}}= 2500$.

To measure the similarity degree between the support sets of two vectors, we introduce the {\it identification accuracy} (IAcc) function.  Given two vectors ${\h x}^{(1)}$, ${\h x}^{(2)}\in \mathbb R^n$, we define
\begin{eqnarray*}
	\text{IAcc}(\h x^{(1)},\h x^{(2)}) = \frac{1}{n}\sum_{i = 1}^{n}\delta(\h x^{(1)}_{i},\h x^{(2)}_{i})
\end{eqnarray*}
where $\delta(\h x^{(1)}_{i},\h x^{(2)}_{i}) = 1$ if both $\h x^{(1)}_{i}$ and $\h x^{(2)}_{i}$ are either zero or nonzero; otherwise, $\delta(\h x^{(1)}_i,\h x^{(2)}_{i}) = 0$. Obviously, $\text{IAcc}(\h x^{(1)},\h x^{(2)})=\text{IAcc}(\h x^{(2)},\h x^{(1)})$ and $0 \leqslant \text{IAcc}(\h x^{(1)},\h x^{(2)})\leqslant 1$. A larger  IAcc value signifies higher accuracy in the identification. Clearly, $\text{IAcc}(\h x^{(1)},\h x^{(2)})=1$ if and only if ${\text{supp}}({\h x}^{(1)})={\text{supp}}({\h x}^{(2)})$.

To measure the extent of satisfying optimality condition of (\ref{L1o2uncon}), we define the Karush-Kuhn-Tucker residual on  the  support set (KKT$_R$) of the  last iterate ${{\h x}}$ as:
 \begin{eqnarray*} \text{KKT$_R$} = \left\| \gamma\left(\frac{{\text{sign}}({{\h x}}_{{\Lambda}})}{\|{{\h x}}\|_2}-\frac{\|{{\h x}}\|_1}{\|{{\h x}}\|_2^3}{{\h x}}_{{\Lambda}}\right)+(A_{{\Lambda}})^\top(A{{\h x}}-{\h b})\right\|_2,   \end{eqnarray*}
where ${\Lambda}=\operatorname{supp}({\h {x}})$.

\subsection{Sparse Recovery} \label{SR} We focus on the sparse recovery
by solving (\ref{L1o2uncon}) with $\Phi({\h x})=\frac{1}{2}\|A{\h x}-{\h b}\|_2^2$ and $A \in \mathbb{R}^{m \times n}\;(m \ll n)$, ${\h b}\in\mathbb R^m$.
Two types of sensing matrices are considered:
 (I) Gaussian matrix.  $A$ is subject to ${\cal N}({\bf 0}, \Sigma)$  with
	$\Sigma=((1-r)I_n+r)$
	with $r\in(0,1)$.
(II)  Oversampled DCT (O-DCT).  $A = [{\bf a}_1,{\bf a}_2,\ldots,{\bf a}_n]\in {\mathbb R}^{m\times n}$
	with each column ${\bf a}_j=\frac{1}{\sqrt{m}}\cos \left( \frac{2\pi {\bf w}j}{F}\right)(j=1,\ldots,n),$
	where ${\bf w}\in{\mathbb R}^m$ is an uniformly distribution on $[0,1]$ random vector  and
	$F\in {\mathbb R}_+$ controls the coherence.
We generate an $s$-sparse ground truth signal ${\h x}^*\in{\mathbb R}^n$  with
 dynamic range (${\tt sign(randn(K,1)).*10^D.*randn(K,1)}$ in matlab script) and ${\tt D}$
 is specified later.

First, we test two types of matrices (Gaussian matrix, O-DCT) with the ground-truth having sparsity $s=12$, where nonzero entries are generated by dynamic range with ${\tt D}=1$. We generate ${\h b}= A{\h x}^*$. The size of the sensing matrix $(m,n)$ is $256\times 2048$. We set $\gamma=0.0001$ in (\ref{L1o2uncon}) and $\beta=0.015$ in ADMM$_p$.

In Fig. \ref{smooth}, we depict the evolution of RErr $\left(\frac{\|{\h x}^k-{\h x}^{*}\|_2}{\|{\h x}^{*}\|_2}\right)$ and KKT$_R$ over iterations for ADMM$_p$ and HAFAM$_T$ with different values of $T$ (i.e., $T=5,10,20,30$). The results are presented separately for the cases of Gaussian matrices (top row) and O-DCT matrices (bottom row).
Notably, it is evident that the convergence of HAFAM$_T$ (particularly with $T=5$) is significantly faster compared to ADMM$_p$. The RErr and KKT$_R$ values consistently exhibit a linear reduction, leading to the attainment of the lowest values. This accelerated convergence process stands out as a significant advantage of HAFAM$_T$.

In Table \ref{Table1.2}, we present the outcomes of ADMM$_p$ (AD$_p$) and HAFAM$_T$ (HA$_T$) concerning the Gaussian matrix scenario in terms of  RErr, computational time in seconds (CPU), final objective function value (Obj), KKT$_R$, the iteration number of transferring to {\bf SSNewton} (TranIt), total iteration number (ToIt), $\text{IAcc}({\h x}^{I},\check{{{\h x}}})$ (IAcc1).
Here, ${\h x}^{I}$ represents the iterate from HAFAM$_T$ when the heuristic strategy (\ref{adjLam}) is satisfied. $\check{{{\h x}}}$ corresponds to the last iterate from ADMM$_p$ until the stopping criterion (\ref{StopC}) is satisfied. The data in Table \ref{Table1.2} illustrates that HAFAM$_T$ identifies the active manifold (i.e., the true support set) within a finite number of iterations for all tested $T$, as indicated by the values of IAcc1 being all $1$. Notably, ADMM$_p$ also achieves this within a finite iteration, as the first phase in HAFAM$_T$ is essentially ADMM$_p$.
\begin{figure}[t]
	\vspace{0cm}\centering{\vspace{0cm}
		\begin{tabular}{cc}
		\includegraphics[scale = .32]{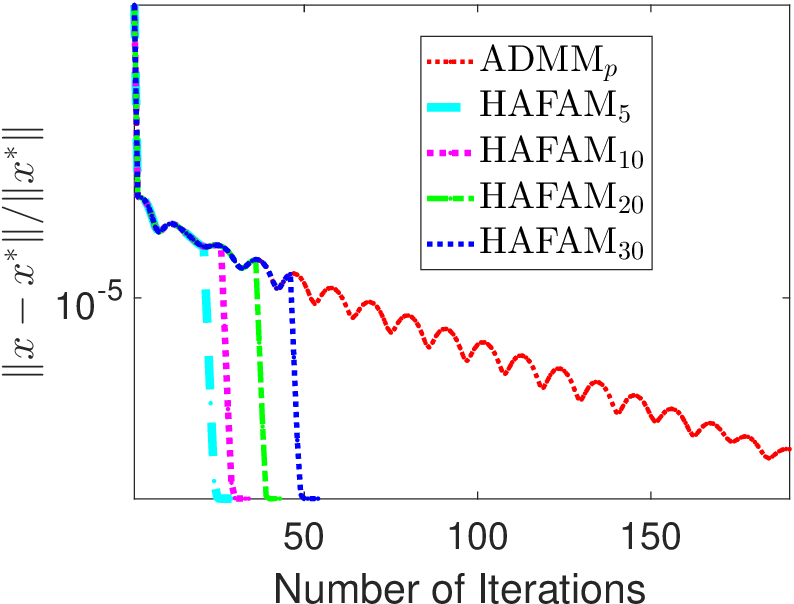}\!\!&\!\!  \includegraphics[scale = .32]{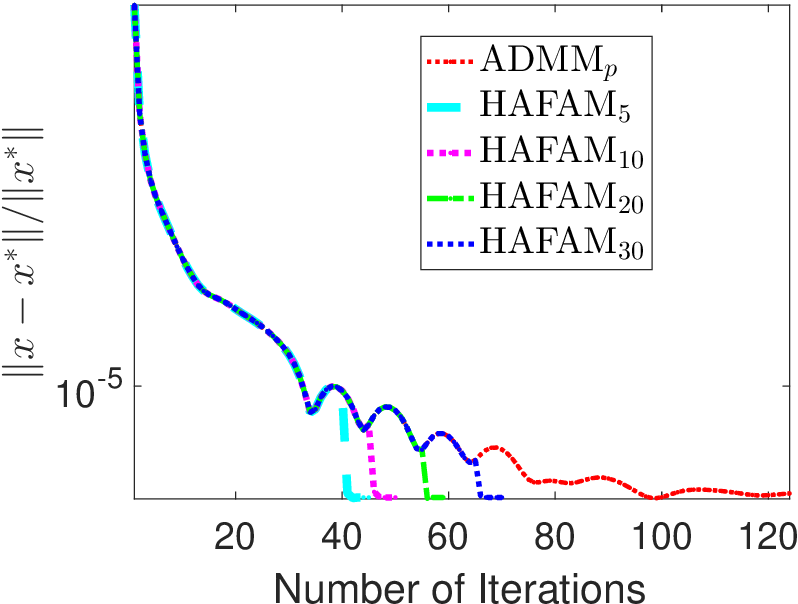} \\
\includegraphics[scale = .32]{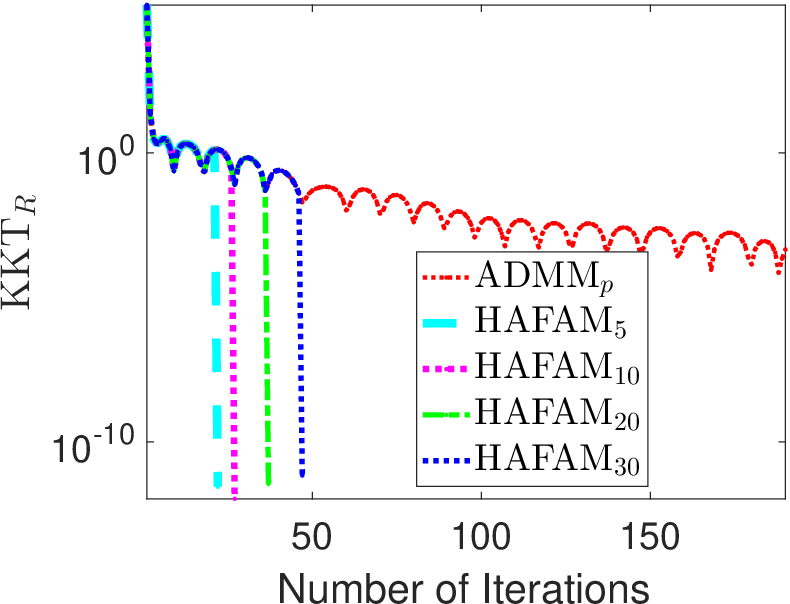}\!\!&\!\!\includegraphics[scale = .32]{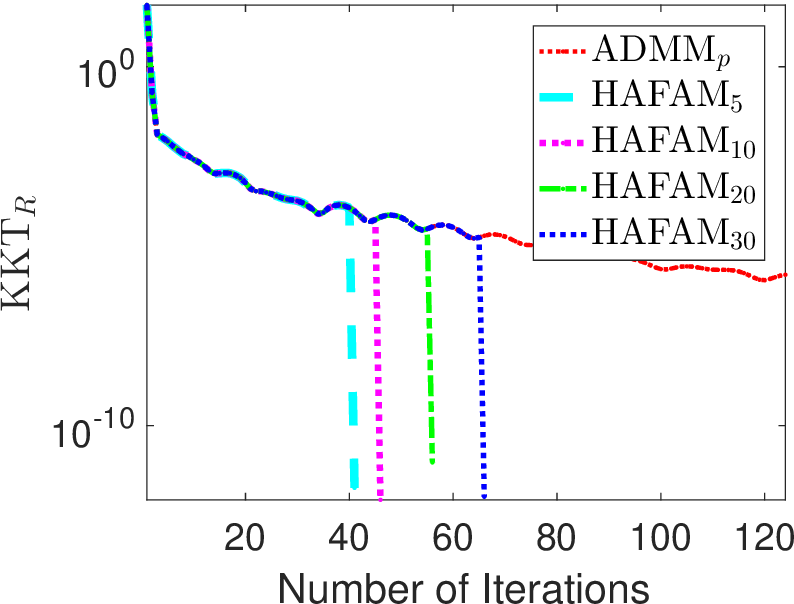}\\
		\end{tabular}
	} \caption{
Comparison between ADMM$_p$ and HAFAM$_T$ ($T=5,10,20,30$)  for Gaussian matrices with $r=0.8$ (top row) and O-DCT matrices with $F=10$ (bottom row). The evolution of RErr and KKT$_R$ from HAFAM$_T$ consistently exhibits a linear reduction when the active manifold is identified, highlighting its accelerated convergence over ADMM$_p$.}
	\label{smooth}\end{figure}

 \begin{table}[tt]
	\begin{center}
		\caption{Results on ADMM$_p$ and HAFAM$_T$}\label{Table1.2}
		\vspace{-0.2cm}\begin{tabular}{c|ccccc}
			\hline
{Algo.} \!\!&\!\!  AD$_p$  & HA$_5$  & HA$_{10}$  & HA$_{20}$    & HA$_{30}$ \\
			\hline
           RErr  &4.45e-08&7.59e-09&7.59e-09&7.59e-09&7.59e-09\\

           CPU &3.58e-01&5.27e-02&5.71e-02&9.19e-02&9.68e-02\\

            Obj  &2.99e-04&2.99e-04&2.99e-04&2.99e-04&2.99e-04\\

            KKT$_R$  &4.55e-04&3.09e-12&1.07e-12&3.39e-12&7.08e-12\\

            TranIt   & ----       &21      &26      &36      &46\\

            ToIt       & 190    &30      &35      &44      &55\\

            IAcc1    & 1     & 1       & 1       &  1      & 1\\

            \hline
		\end{tabular}
	\end{center}
\end{table}

In the second part, we perform a performance profiles \cite{DM02} for ADMM$_p$ and HAFAM$_T$ ($T=5,10,20,30$) in terms of KKT$_R$ and CPU. The analysis involves {\it $90$ different problems} with Gaussian matrices ($r=0.7,0.8,0.9$) and O-DCT matrices ($F=5,10,15$), varying sparsity ($s=4:2:32$), and dimensions $(m,n)=(64,1024)$. Nonzero entries of the ground truth are generated with dynamic range using ${\tt D}=2$.

Let $t_{p,j}$ denote a performance metric (lower are preferable) for the $j$th solver on problem $p$.
 We calculate the ratio $r_{p,j}$ by dividing $t_{p,j}$ by the smallest value achieved by any of the $j$ solvers for problem $p$, i.e.,
$r_{p,j}=\frac{t_{p,j} }{\min\{t_{p,j}:1\le j\le n_j\}}$.
For a given threshold $\tau>0$, $\pi_j(\tau)$ is  the ratio of the number of problems where $r_{p,j}\le\tau$ to the total number.
 This provides insight into whether solver $j$ performs within a factor of $\tau$ in comparison to the best-performing solver.

 Fig. \ref{performance}  presents the performance profiles of KKT$_R$ and CPU.
It shows that   KKT$_R$  of HAFAM$_{30}$ is better than the others on all the problems while
CPU time of HAFAM$_{5}$ are less than the others on the most problems.
A similar phenomenon is also observed for RErr, we omit the corresponding plot  for conciseness.
 \begin{figure}[t]
	\vspace{0cm}\centering{\vspace{0cm}
		\begin{tabular}{cc}
		\includegraphics[scale = .32]{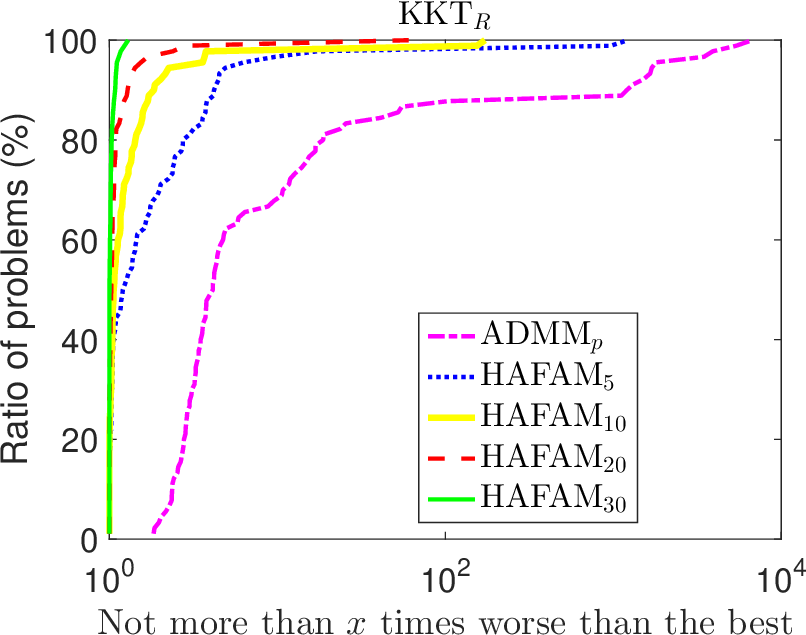}\!\!& \!\!\includegraphics[scale = .32]{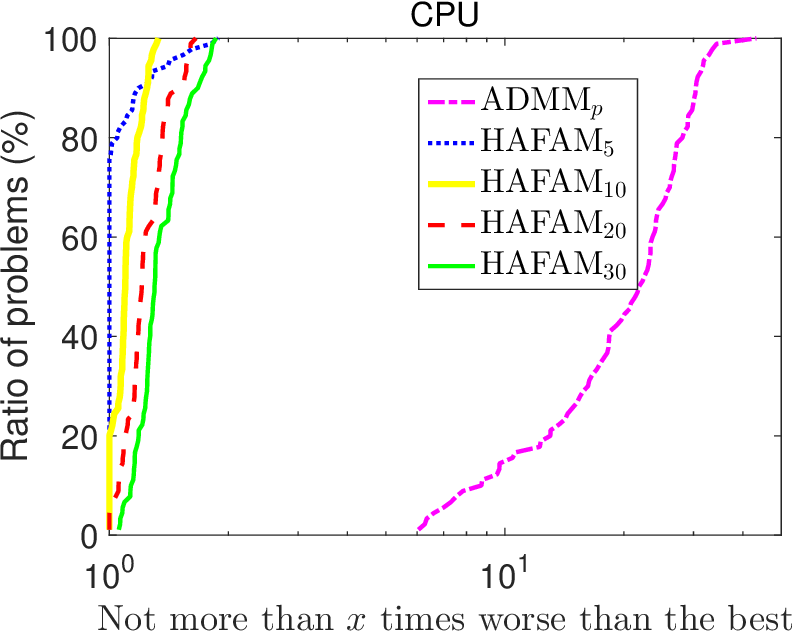} \\
		\end{tabular}
	} \caption{The performance profiles of ADMM$_p$ and HAFAM$_T$ ($T=5,10,20,30$) for {\it $90$ different problems}.
HAFAM$_{30}$ excels in KKT$_R$ performance among these comparing algorithms, while HAFAM$_{5}$ takes the least CPU time.
}
	\label{performance}\end{figure}

 \begin{figure}[t]
	\vspace{0cm}\centering{\vspace{0cm}
		\begin{tabular}{cc}
		\includegraphics[scale = .32]{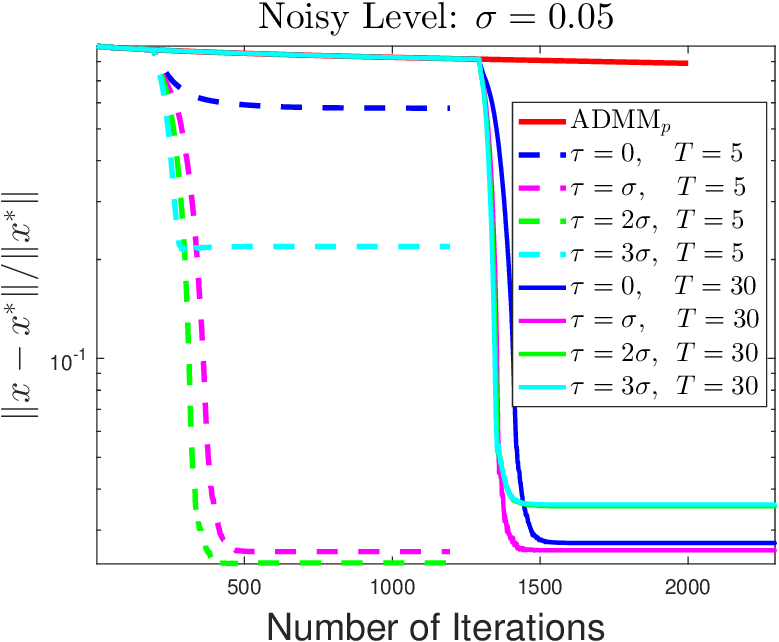}\!\!& \!\!\includegraphics[scale = .32]{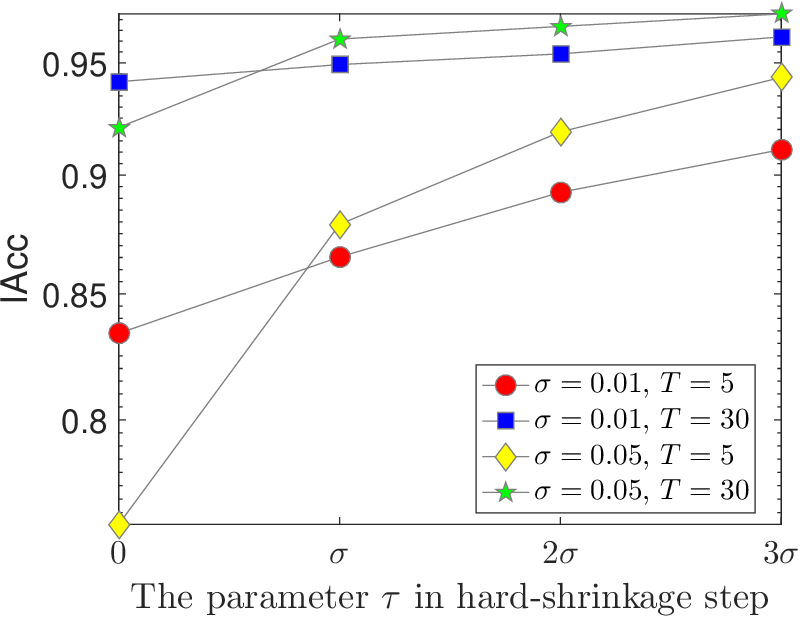} \\
		\end{tabular}
	} \caption{The evolution of the RErr from ADMM$_p$ and HAFAM$_T$ ($T = 5,\;30$) with varying $\tau$ values under a noise level of $\sigma = 0.05$ (left) and the $\text{IAcc}({\h x}^{I},{{{\h x}^*}})$ values for different $\tau$ settings under two noise levels (right). Specifically, the values of $\text{IAcc}({\h x}^{I},{{{\h x}^*}})$ for $\tau = 0,\sigma,2\sigma,3\sigma$ are provided: (a) [$\sigma = 0.01, T = 5$]: 0.83, 0.87, 0.89, 0.91; (b) [$\sigma = 0.01, T = 30$]: 0.94, 0.95, 0.95, 0.96; (c) [$\sigma= 0.05, T = 5$]: 0.76, 0.88, 0.92, 0.94; (d) [$\sigma = 0.05, T = 30$]: 0.92, 0.96, 0.97, 0.97.
}
	\label{HardS}\end{figure}


In summary, Phase I of Algorithm \ref{alg:ACT} is dedicated to identifying the active manifold, and Phase II focuses on accelerating the convergence rate within the dimension-reduced space. The introduction of the hard-shrinkage step (Line 9) in Algorithm \ref{alg:ACT} is specifically to address the noisy data.

To empirically substantiate this, we conduct a case study for the model (\ref{L1o2uncon}) with $\Phi({\h x})=\frac{1}{2}\|A{\h x}-{\h b}\|_2^2$. The matrix $A \in \mathbb{R}^{64 \times 1024}$ is O-DCT, and the ground truth ${\h x}^*\in{\mathbb R}^{1024}$ has sparsity $s=6$ and a dynamic range ${\tt D}=1$. We generate ${\h b}=A{\h x}^*+\varepsilon$, where the noise $\varepsilon$ follows a normal distribution $N({\bf 0},\sigma)$, with two noise levels: (a) $\sigma=0.01$ and (b) $\sigma=0.05$.

The evolution of RErr over iterations is provided in Fig. \ref{HardS} (left plot) for the case $\sigma=0.05$. We present results for ADMM$_p$ and HAFAM$_T$ with $T=5,30$, both with and without the hard-shrinkage step, corresponding to $\tau=\sigma,2\sigma,3\sigma$, and $\tau=0$. Notably, the inclusion of the hard-shrinkage step in HAFAM$_5$ at $\tau=2\sigma$ outperforms the others, demonstrating the lowest RErr. As $T=30$, the performance gap between with and without the hard-shrinkage step diminishes compared to $T=5$.

Additionally, we compute the value of IAcc(${\h x}^{I},{{{\h x}^*}})$, where ${\h x}^{I}$ represents the iterate from HAFAM$_T$ satisfying (\ref{adjLam}), and ${\h x}^*$ is the ground truth. The results are plotted in Fig. \ref{HardS} (right plot). We observe an increase in the accuracy of identifying the true support set when incorporating the hard-shrinkage step under both noisy levels.

\subsection{Validation of finite identification} \label{ValFI}
We conduct the numerical validations for the finite identification property of ADMM$_p$ and focus on  (\ref{L1o2uncon}) with $\Phi(\h x)=\frac{1}{2}\|A{\h x}-{\h b}\|_2^2$.

We test on Gaussian matrix $A$
   with
$r=0.8$. The dimensions of the sensing matrix
$(m,n)$ are fixed at $n=1024$, with  $m=16:16:256$.
We generate an
$s$-sparse ground truth signal  ${\h x}^*\in{\mathbb R}^n$  with
 dynamic range with ${\tt D}=1$ and let the sparsity $s=1:1:16$. This results in a total of
256 distinct problem instances.

For each of the 256 distinct problems, we execute ADMM$_p$ until (\ref{StopC}) is satisfied, recording the final iterate as ${\hat{\h x}}$. Subsequently, we apply HAFAM$_T$ ($T=5,10,20,30$) until reaching the iterate that satisfies (\ref{adjLam}), denoted as ${\h x}_T^{I}$ (with the subscript $T$ to differentiate different HAFAM$_T$). For each scenario, 50 random instances are generated, and the average ${\text{IAcc}}({\hat{\h x}},{\h x}_T^{I})$ is computed. The results are illustrated in Fig. \ref{FiniteI}, where the $x$-axis and $y$-axis represent ratios of $s/m$ and $m/256$, respectively.

As observed in Fig. \ref{FiniteI}, the value of ${\text{IAcc}}({\hat{\h x}},{\h x}_T^{I})$ consistently increases with the growth of $T$, ranging between 0.98 and 1. This indicates that running HAFAM$_T$ until it satisfies (\ref{adjLam}) can identify the support set (i.e., active manifold).
It also implies that ADMM$_p$  can identify the active manifold within a finite number of iterations.


\begin{figure}[t]
	\vspace{0cm}\centering{\vspace{0cm}
		\begin{tabular}{cc}
		\includegraphics[scale = .285]{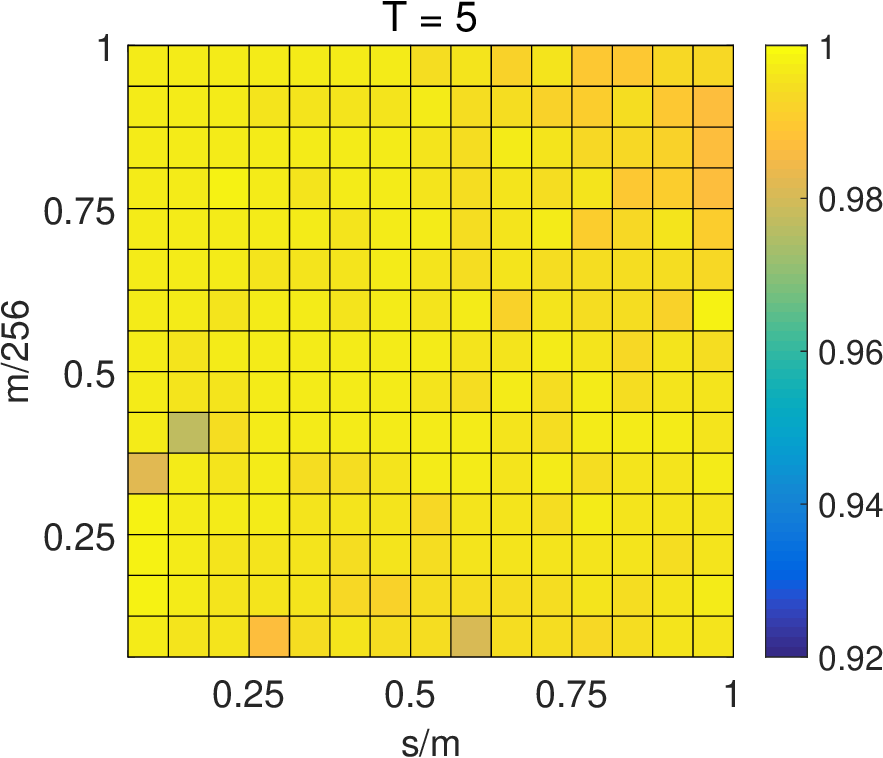}\!\!\!&\!\!\! \includegraphics[scale = .285]{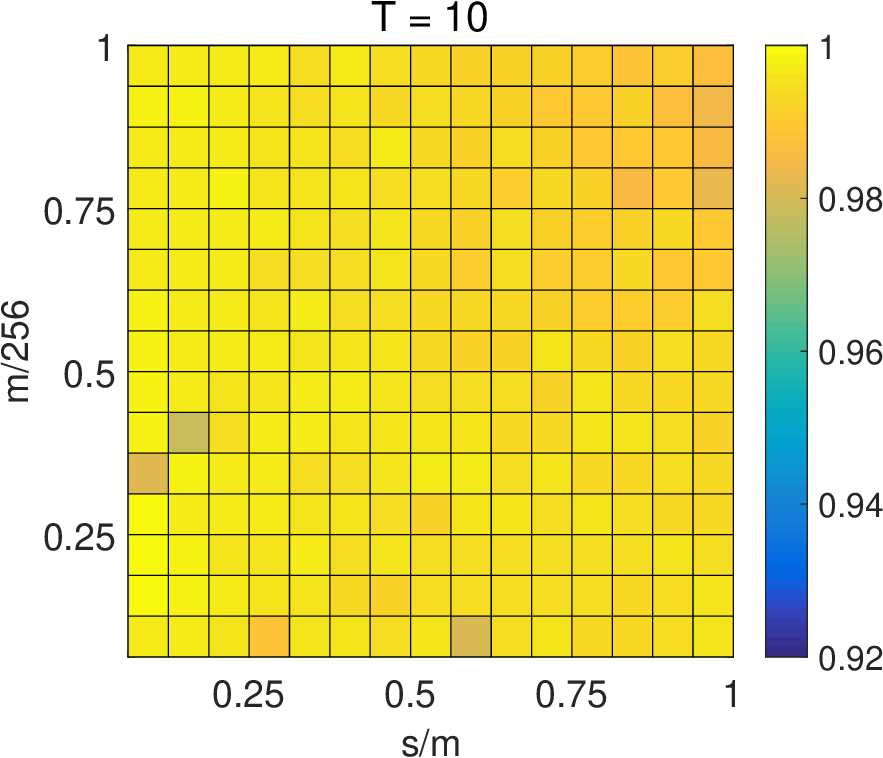} \\
\includegraphics[scale = .285]{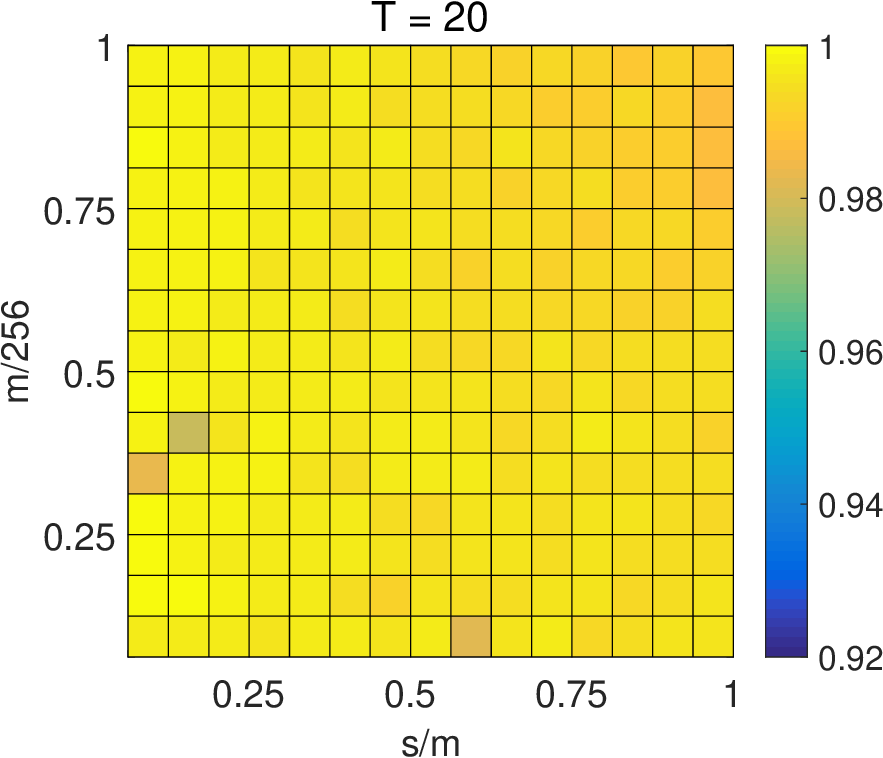}\!\!\!&\!\!\! \includegraphics[scale = .285]{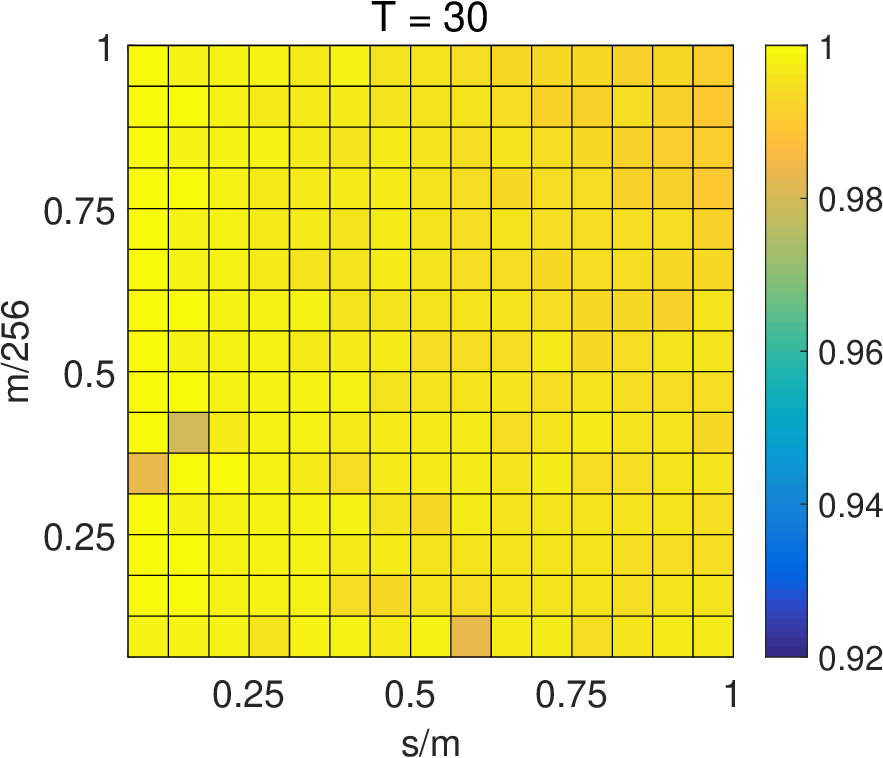} \\
		\end{tabular}
	} \caption{Each plot depicts the average of ${\text{IAcc}}({\hat{\h x}},{\h x}_T^{I})$ across 50 random instances for 256 different problems, where $n=1024$, $m=16:16:256$, and $s=1:1:16$.
  Top row: $T=5$ (left) and $T=10$ (right); Bottom row: $T=20$ (left) and $T=30$ (right).
 For $T=5,10,20,30$, the minimum value of ${\text{IAcc}}({\hat{\h x}},{\h x}_T^{I})$ in each plot  are $0.98,0.98,0.98,0.98$;
 and the maximum value of ${\text{IAcc}}({\hat{\h x}},{\h x}_T^{I})$  in each plot are  $1.00,1.00,1.00,1.00$.
}
	\label{FiniteI}\end{figure}

\subsection{Realistic Dataset}\label{realdata}
We verify the superiority of
 the proposed algorithm (HAFAM$_{T}$) in comparison with ADMM$_p$ through $L_1/L_2$ model (\ref{L1o2uncon}) with two types of data fidelity term
 $\Phi(\h x)=\frac{1}{2}\|A{\h x}-{\h b}\|_2^2$ or $\|A{\h x}-{\h b}\|_2$ on {\it real-world} datasets.
  Simultaneously, we also implement other sparse recovery models, such as $L_1$ and $L_{1/2}$ \cite{XuChang12}.
 As for $L_1$, we solve it by
ADMM.
In particular, we adopt the following datasets:
\begin{itemize}
	\item Diabetes dataset \footnote{\url{https://www4.stat.ncsu.edu/~boos/var.select/diabetes.html}};
	\item UCI standard dataset \footnote{\url{http://archive.ics.uci.edu/datasets}}: Prostate cancer, Pyrim, Servo, Mpg, Energy and SkillCrift1;
	\item StatLib datasets archive \footnote{\url{http://lib.stat.cmu.edu/datasets/}}: Space-ga.
\end{itemize}

\noindent Each dataset consists of an observation matrix $A\in\mathbb{R}^{m\times n}$ along with the response vector ${\h  b}\in\mathbb{R}^{m\times 1}$
and  we pre-processing each data (i.e., $A$ and ${\h b}$) to have mean $0$ and squared length $1$ in column sense \cite{hastie2009elements}.

 To evaluate the performance of various algorithms, we divided the dataset randomly into training and test sets with a specific ratio (${\tt ratio}$). After solving the corresponding problem of (\ref{L1o2uncon}) on the training sets, we evaluate the Mean Squared Error  via the test set and define it as
 \begin{eqnarray*}\label{TMSE}
	\text{TMSE} = \frac{1}{{N}_{\text{test}}}|| A_{\text{test}}{\h x} - {\h b}_{\text{test}}||_2^{2},
\end{eqnarray*}
 where $A_{\text{test}}, {\h b}_{\text{test}}$ denote the corresponding observations and response. $N_{\text{test}}$ is the number of test data.

For dealing with real data, we replace Line 9 in Algorithm \ref{alg:ACT} with the following scripts:
\begin{itemize}
 \item Sort ${\h {x}}^{I}$ in ascending order based on absolute values, denoted as ${\h x}^{\uparrow}$;\\[-0.4cm]
\item Define ${\bar i}\in\arg\max\{i\in[n]|\sum_{j=1}^i |x_j^\uparrow |<\tau\|{\h x}^\uparrow\|_1\}$;\\[-0.3cm]
 \item Compute ${\widehat{\h x}}={\text{HARD}}({\h x}^{I},{\h x}^{\uparrow}_{\bar i})$.
 \end{itemize}

In the initial set of experiments, we conducted tests on the model (\ref{L1o2uncon}) with $\Phi(\h x)=\frac{1}{2}\|A{\h x}-{\h b}\|^2$ by using the `Diabetes' dataset with two different ratios: (a) {\tt ratio}= 8:2 and (b) {\tt ratio}= 7:3. For all algorithms tested, we employed a 10-fold cross-validation on the training set to determine the parameter $\gamma$ in  (\ref{L1o2uncon}).

\begin{figure}[t]
	\vspace{0cm}\centering{\vspace{0cm}
		\begin{tabular}{cc}
 		\includegraphics[scale = .3]{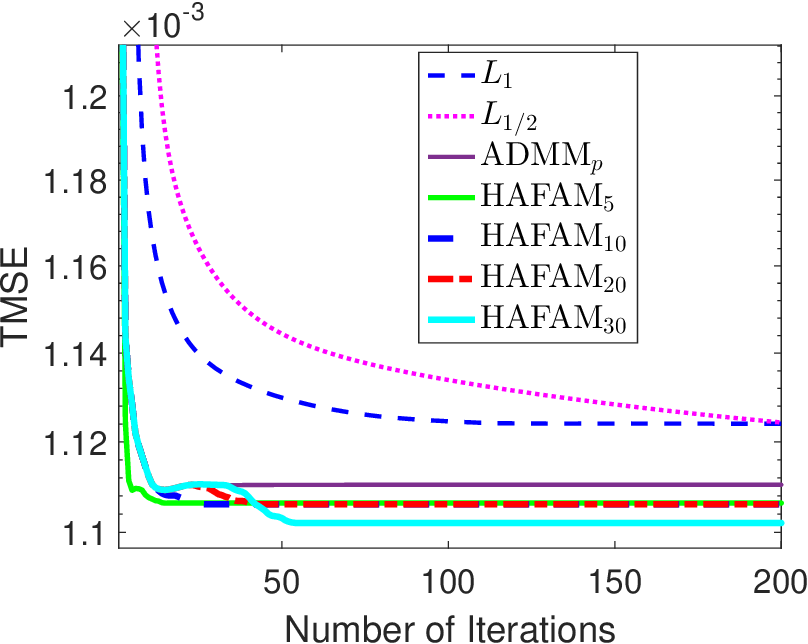}& \includegraphics[scale = .3]{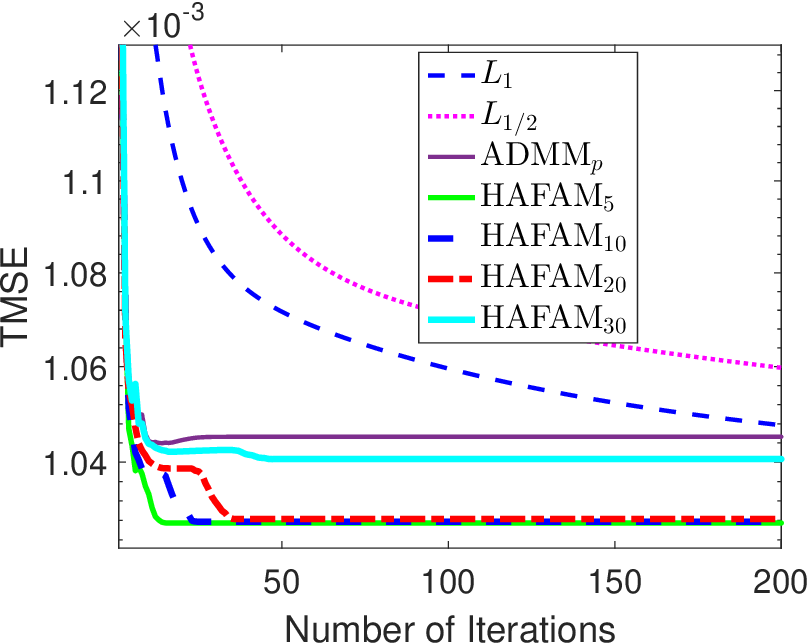} \\
\includegraphics[scale = .3]{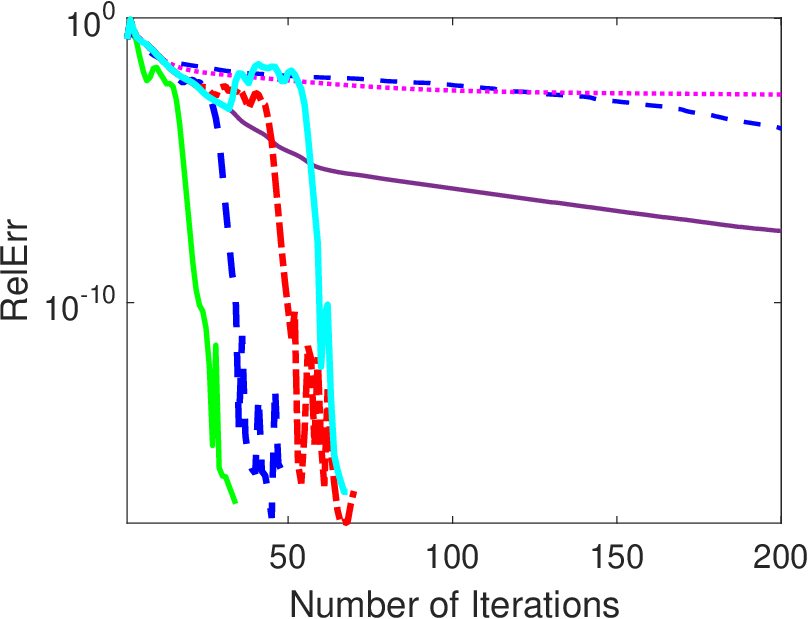} &\includegraphics[scale = .3]{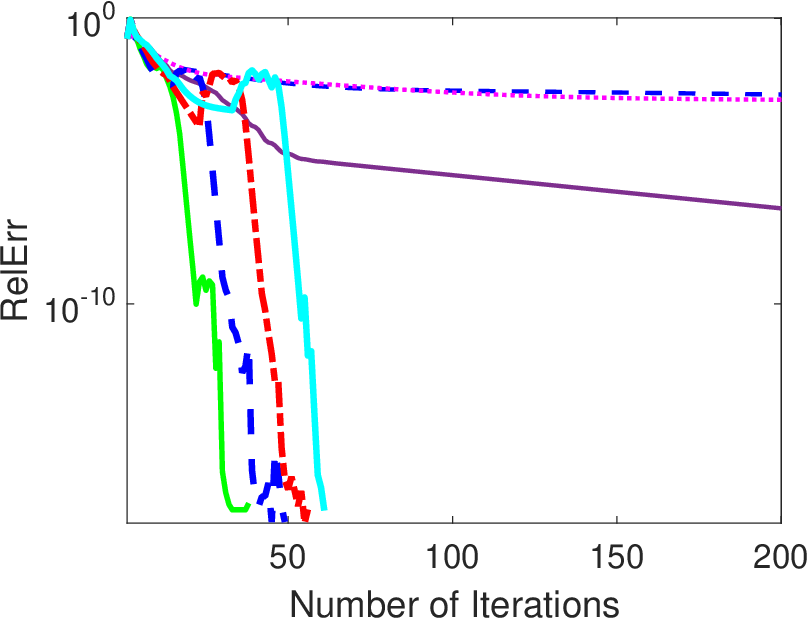}
		\end{tabular}
	} \caption{Comparison among $L_1/L_2$ via ADMM$_p$ and HAFAM$_T$, $L_1$ and $L_{1/2}$. The evolution of test MSE (TMSE) and
relative error (RelErr) on diabetes dataset: {\tt ratio} = $8:2$ (left); {\tt ratio} =
$7:3$ (right). ADMM$_p$ and HAFAM$_T$ outperform $L_1$ and $L_{1/2}$ in terms of
TMSE and RelErr. While ADMM$_p$ and HAFAM$_T$ initially overlap in TMSE
and RelErr, HAFAM$_T$ quickly converges to a smaller value when the active
manifold is identified.
}
	\label{RelMSE}\end{figure}

For real datasets, the true solutions are {\it unknown}. Instead of utilizing RErr, we use Relerr (\ref{StopC}) to evaluate the performance. Fig. \ref{RelMSE} illustrates the variation of Relerr and TMSE with respect to the number of iterations for each algorithm. All these curves represent the average of 100 runs.  All algorithms utilized Gaussian random initialization with the {\tt randn} function.

The first row of Fig. \ref{RelMSE} displays the TMSE history for each method on the test set under two different ratios ($8:2$ and $7:3$). It is evident that both ADMM$_p$ and HAFAM$_{T}$ outperform $L_{1}$ and $L_{1/2}$ in terms of TMSE.
Furthermore, HAFAM$_{T}$ rapidly converges to a solution with a smaller TMSE as soon as the semismooth Newton method starts.
 For HAFAM$_{T}$, an increase in the parameter $T$ increases iterations. This  occurs because HAFAM$_{T}$
needs more iterations to identify the active manifold.

As can be seen, HAFAM$_{30}$ converges to the lowest value of TMSE for the scenario with ${\tt ratio}=8:2$. For other values of $T$ (i.e., $T=5,10,20$), HAFAM$_{T}$ converges to almost the same TMSE under both ratios. However, for the ${\tt ratio}$ of $7:3$, HAFAM$_{30}$ converges to a slightly higher TMSE value compared to the other HAFAM$_{T}$ $(T=5,10,20)$, which might be a result of an over-fitting support set.

 Regarding RelErr, the curves of ADMM$_p$ and HAFAM$_{T}$ overlap in the early phases.
$L_1$ and $L_{1/2}$ perform almost the same and much worse than ADMM$_p$ and HAFAM$_{T}$. HAFAM$_T$ consistently achieves a smaller RelErr at the termination point compared to the others, and it converges superlinearly after identifying the active manifold.

Next, we conducted numerical simulations for the $L_1/L_2$ model (\ref{L1o2uncon}), under two types of data fitting term
  $\Phi(\h x)=\frac{1}{2}\|A{\h x}-{\h b}\|_2^2$ and $\Phi(\h x)=\|A{\h x}-{\h b}\|_2$ via HAFAM$_{T}$ (HA$_T$) and ADMM$_p$ (AD$_p$)   comparing with the corresponding $L_1$ and $L_{1/2}$ regularization models. The simulations were conducted across various real datasets with a split ratio of $8:2$. With fine-tuned parameter $\tau$ for each scenario, the results are summarized in Tables \ref{table1}-\ref{table2}
  for $\Phi(\h x)=\frac{1}{2}\|A{\h x}-{\h b}\|_2^2$ and Tables \ref{LR1}-\ref{LR2} for $\Phi(\h x)=\|A{\h x}-{\h b}\|_2$, respectively.

The results presented in Tables \ref{table1},\ref{table2} and \ref{LR1},\ref{LR2} are obtained from an average of $100$ independent runs for each instance, and performance is measured in terms of TMSE with standard deviation (${\tt std}$ in parentheses), the sparsity of the solution (${\tt nnz}$), and CPU time (CPU) in seconds.
The results in these tables correspond to two different initialization strategies: (1) Gaussian random initialization (${\tt randn}$) (Table \ref{table1} and \ref{LR1}) and (2) zero initialization (Table \ref{table2} and \ref{LR2}).

As observed from Table \ref{table1},\ref{table2} and \ref{LR1},\ref{LR2},
HAFAM$_T$ with the $L_1/L_2$ model is a promising approach, offering improved accuracy and sparsity compared to traditional $L_1$ and $L_{1/2}$ regularization.
There is minimal difference in TMSE between different initializations for HAFAM$_T$ and ADMM$_p$, highlighting robustness. Increasing $T$ tends to decrease TMSE and induce sparser solutions. HAFAM$_{30}$ performs the best across most datasets, achieving the lowest TMSE, and HAFAM$_T$ ($T=5,10,20$) shows competitive results.

Interestingly, even with small $T$, HAFAM$_T$ produces satisfactory results for real datasets. Although increasing $T$ raises TMSE, the impact is not significant, and outcomes depend on specific datasets. Generally, a smaller $T$ in HAFAM$_T$ results in reduced CPU time, while larger $T$ values lead to smaller TMSE. Therefore, it is crucial in practice to carefully select an appropriate value for $T$ in HAFAM$_T$.


\begin{table}[t]
	\centering
	\caption{Average results on real datasets from {\tt randn} initialization with data fidelity term $\Phi(\cdot)=\|\cdot\|^2_2$.
HAFAM$_T$ achieves much lower  TMSE and much sparser solutions compared with others.}\label{table1}
	{ \scriptsize\vskip -1mm\begin{tabular}{clccc}
		\hline
		Dataset               & \multirow{2}{*}{Algo. ($\tau$)} & \multirow{2}{*}{TMSE ({\tt std})}  & \multirow{2}{*}{{\tt nnz}}  & \multirow{2}{*}{CPU} \\
		(m,n)                &                                  &                                 &                   &
		\\ \hline
		\multirow{5}{*}{\begin{tabular}[c]{@{}c@{}}Diabetes\\ (422,10)\end{tabular}}
		& $L_1$	 & 1.124e-03 (1.348e-12) &	 8.00 	 & 6.12e+00 	\\
& $L_{1/2}$	 & 1.110e-03 (1.430e-07) &	 9.95 	 & {\bf 4.84e+00} 	 \\
& AD$_{p}$	 & 1.110e-03 (3.672e-11) &	 8.00 	 & 6.44e+00 	\\
& HA$_{5}\;$ (0.15)	 & 1.106e-03 (6.971e-14) &	 7.00 	 & 3.36e+01 	\\
& HA$_{10}$ (0.10)	 & 1.106e-03 (7.271e-14) &	 7.00 	 & 2.70e+01 	\\
& HA$_{20}$ (0.10)	 & 1.106e-03 (1.300e-14) &	 7.00 	 & 3.42e+01 	\\
& HA$_{30}$ (0.18)	 & {\bf 1.102e-03} (6.192e-14) &	{\bf 6.00} 	 & 3.28e+01 	 \\ \hline

		\multirow{5}{*}{\begin{tabular}[c]{@{}c@{}}Prostate  \\cancer\\ (97,8)\end{tabular}}
		& $L_1$	             & 9.542e-01 (5.048e-10)       &	 5.00 	     & 5.28e-01 	\\
        & $L_{1/2}$          & 8.239e-01 (3.028e-03)       &	 2.02 	     & \bf{1.29e-01} 	\\
        & AD$_{p}$	         & 7.579e-01 (2.500e-09)       &	 6.00 	     & 1.92e+00 	\\
        & HA$_{5}\;$ (0.15)	 & {\bf7.680e-01} (7.510e-11)       &	 \bf{3.00} 	 & 6.64e+00 	\\
        & HA$_{10}$ (0.15)	 & {\bf7.680e-01} (4.184e-11)       &	 \bf{3.00} 	 & 5.61e+00 	\\
        & HA$_{20}$ (0.15)	 & {\bf7.680e-01} (4.149e-11)  &	 \bf{3.00} 	 & 6.53e+00 	 \\
        & HA$_{30}$ (0.15)	 & {\bf7.680e-01} (4.460e-11)       &	 \bf{3.00} 	 & 7.76e+00 	\\ \hline
		\multirow{5}{*}{\begin{tabular}[c]{@{}c@{}}Pyrim\\ (74,27)\end{tabular}}
	    & $L_1$	             & 2.528e-03 (8.748e-13)       &	 20.00 	         & \bf{1.58e+00} 	\\
        & $L_{1/2}$          & 2.545e-03 (9.937e-05)       &	 18.78 	         & 4.45e+00 	\\
        & AD$_{p}$	         & 2.457e-03 (5.152e-13)       &	 22.00 	         & 1.17e+01 	\\
        & HA$_{5}\;$ (0.08)	 & 2.405e-03 (4.192e-05)       &	 \bf{15.29} 	 & 6.31e+00 	 \\
        & HA$_{10}$ (0.08)	 & 2.353e-03 (7.696e-05)       &	 15.99 	         & 5.71e+00 	 \\
        & HA$_{20}$ (0.08)	 & {\bf2.310e-03} (9.464e-14)  &	 16.00 	         & 9.19e+00 	 \\
        & HA$_{30}$ (0.08)	 &  {\bf2.310e-03 }(2.638e-13)       &	 16.00 	         & 5.82e+00 	 \\ \hline
		
		\multirow{5}{*}{\begin{tabular}[c]{@{}c@{}}Servo\\ (167,19)\end{tabular}}
		& $L_1$	             & 6.634e-01 (1.147e-07)       &	 16.00 	         & \bf{1.94e+00} 	\\
        & $L_{1/2}$          & 6.631e-01 (9.635e-05)       &	 17.60 	         & 4.63e+00 	\\
        & AD$_{p}$	         & 6.631e-01 (1.303e-07)       &	 16.00 	         & 1.10e+01 	 \\
        & HA$_{5}\;$ (0.06)	 & 6.637e-01 (2.027e-02)       &	 13.73 	         & 4.72e+01 	\\
        & HA$_{10}$ (0.08)	 & 6.715e-01 (2.066e-02)       &	 \bf{12.93} 	 & 3.59e+01 	 \\
        & HA$_{20}$ (0.06)	 & {\bf6.621e-01} (1.689e-02)  &	 13.97 	         & 3.67e+01 	\\
        & HA$_{30}$ (0.06)	 & 6.629e-01 (2.003e-02)       &	 13.67 	         & 3.88e+01 	\\ \hline
		\multirow{5}{*}{\begin{tabular}[c]{@{}c@{}}Mpg\\ (392,7)\end{tabular}}
		& $L_1$	             & 1.003e-03 (2.656e-11)       &	 7.00 	        & 6.11e+00 	\\
        & $L_{1/2}$          & 1.003e-03 (8.832e-07)       &	 7.00 	        & 3.36e+00 	\\
        & AD$_{p}$	         & 1.003e-03 (9.494e-13)       &	 7.00 	        & 7.14e-01 	\\
        & HA$_{5}\;$ (0.10)	 & {\bf9.768e-04} (6.970e-17)       &	 \bf{6.00} 	    & \bf{5.20e-01} 	\\
        & HA$_{10}$ (0.10)	 & {\bf9.768e-04} (3.222e-17)       &	 \bf{6.00} 	    & 6.50e-01 	 \\
        & HA$_{20}$ (0.10)	 & {\bf9.768e-04} (4.535e-17)       &	 \bf{6.00} 	    & 8.85e-01 	\\
        & HA$_{30}$ (0.10)	 & {\bf9.768e-04} (1.342e-17)  &	 \bf{6.00} 	    & 1.14e+00 	\\ \hline
		\multirow{5}{*}{\begin{tabular}[c]{@{}c@{}}Space-ga\\ (3107,6)\end{tabular}}
		& $L_1$	             & 1.436e-04 (4.613e-15)         &	 6.00 	       & 8.15e+01 	 \\
        & $L_{1/2}$          & 1.437e-04 (1.286e-06)         &	 5.41 	       & \bf{1.38e+01} 	\\
        & AD$_{p}$	     & 1.455e-04 (3.934e-15)         &	 6.00 	       & 4.67e+01 	\\
        & HA$_{5}\;$ (0.05)	 & 1.421e-04 (3.813e-16)         &	\bf{5.00} 	   & 1.06e+02 	\\
        & HA$_{10}$ (0.05)	 & 1.421e-04 (1.066e-15)         &	\bf{5.00} 	   & 1.22e+02 	 \\
        & HA$_{20}$ (0.05)	 & 1.421e-04 (1.369e-14)         &	\bf{5.00} 	   & 4.94e+02 	 \\
        & HA$_{30}$ (0.05)	 & {\bf1.419e-04} (1.731e-16)    &	 \bf{5.00} 	   & 1.22e+02 	 \\ \hline
		
		\multirow{5}{*}{\begin{tabular}[c]{@{}c@{}}Energy\\ (768,8)\end{tabular}}
		& $L_1$	             & 1.156e-04 (2.557e-12)           &	 7.00 	     & 1.39e+01 	\\
        & $L_{1/2}$          & 1.151e-04 (2.478e-07)           &	 6.91 	     & 5.02e+00 	\\
        & AD$_{p}$	     & 1.156e-04 (8.633e-10)           &	 7.00 	     & \bf{3.90e+00} 	 \\
        & HA$_{5}\;$ (0.10)	 & {\bf 1.142e-04} (3.194e-07)           &	 5.13 	     & 2.22e+01 	 \\
        & HA$_{10}$ (0.10)	 & {\bf1.142e-04} (3.351e-07)      &	 5.12 	     & 1.95e+01 	 \\
        & HA$_{20}$ (0.08)	 & {\bf1.142e-04} (3.352e-07)           &	 \bf{5.11} 	 & 2.87e+01 	 \\
        & HA$_{30}$ (0.08)	 & {\bf1.142e-04} (3.352e-07)           &	 \bf{5.11} 	 & 6.53e+00 	\\ \hline

		\multirow{5}{*}{\begin{tabular}[c]{@{}c@{}}SkillCraft1\\ (3338,18)\end{tabular}}
		& $L_1$	 & 2.294e-04 (2.703e-13) &	 8.00 	 & 2.09e+02 	 \\
& $L_{1/2}$	 & 2.297e-04 (5.905e-07) &	 7.09 	 & {\bf 7.09e+00 }	\\
& AD$_{p}$	 & 2.279e-04 (1.627e-13) &	 8.00 	 & 1.90e+02 	\\
& HA$_{5}\;$ (0.02)	 & 2.275e-04 (1.388e-06) &	 6.83 	 & 5.78e+02 	\\
& HA$_{10}$ (0.02)	 & 2.271e-04 (8.069e-07) &	 6.52 	 & 5.37e+02 	\\
& HA$_{20}$ (0.02)	 & 2.271e-04 (8.087e-07) &	 6.40 	 & 5.48e+02 	 \\
& HA$_{30}$ (0.02)	 & {\bf 2.270e-04} (6.872e-07) &	 {\bf 6.20} 	 & 5.40e+02 	 \\ \hline

	\end{tabular}}
\end{table}

\begin{table}[t]
	\centering
	{ \scriptsize\vskip -2mm\caption{Average results on real datasets from zero initialization with data fidelity term $\Phi(\cdot)=\|\cdot\|^2_2$.
HAFAM$_T$  outperforms the others with lower TMSE and sparser solutions.
The final TMSE from HAFAM$_T$ is not sensitive to $T$.}\label{table2}
	\begin{tabular}{clccc}
		\hline
		 Datase                                           & \multirow{2}{*}{Algo. ($\tau$)}
		 & \multirow{2}{*}{TMSE ({\tt std})}                 & \multirow{2}{*}{{\tt nnz}}     & \multirow{2}{*}{CPU}
		 \\ (m,n)            &                              &                          &                  &
		 \\ \hline
		\multirow{5}{*}{\begin{tabular}[c]{@{}c@{}}Diabetes\\ (422,10)\end{tabular}}
		& $L_1$	     		    & 1.110e-03 (2.935e-12) 	   &	 10.00 	 	 & 2.17e+01 \\
		& $L_{1/2}$	 		    & 1.110e-03 (6.538e-19) 	   &	 10.00 	 	 & 2.72e+00 \\
		& AD$_{p}$ 		        & 1.110e-03 (4.359e-19) 	   &	 8.00 	 	 & 3.81e+00 \\
		& HA$_{5}$\;\;\;(0.10)	& 1.106e-03 (4.359e-19)        &	 7.00 	 	 & \bf{4.50e-01} \\
		& HA$_{10}$	(0.15)      & 1.106e-03 (6.538e-19) 	   &	 7.00 	 	 & 5.25e-01 \\
		& HA$_{20}$	(0.15)      & 1.106e-03 (6.538e-19)        &	 7.00 	 	 & 6.80e-01 \\
		& HA$_{30}$	(0.10)      & {\bf1.102e-03} (6.538e-19)   &	 \bf{6.00} 	 & 8.32e-01 \\ \hline
		\multirow{5}{*}{\begin{tabular}[c]{@{}c@{}}Prostate \\ cancer\\ (97,8)\end{tabular}}
			& $L_1$	 & 9.542e-01 (7.811e-16) &	 5.00 	 & 5.65e-01 	\\
& $L_{1/2}$	 & 8.423e-01 (1.004e-15) &	 {\bf 3.00} 	 & {\bf 1.96e-01 }	\\
& AD$_{p}$	 & 7.579e-01 (8.927e-16) &	 6.00 	 & 2.18e+00 	\\
& HA$_{5}\;$ (0.15)	 & {\bf 7.680e-01} (1.562e-15) &	{\bf 3.00 }	 & 3.79e-01 	\\
& HA$_{10}$ (0.15)	 & {\bf 7.680e-01} (1.116e-15) &	 {\bf 3.00 }	 & 8.81e+01 	\\
& HA$_{20}$ (0.15)	 & {\bf 7.680e-01} (7.811e-16) &	 {\bf 3.00 }	 & 4.94e-01 	\\
& HA$_{30}$ (0.15)	 & {\bf 7.680e-01} (6.695e-16) &	 {\bf 3.00 }	 & 6.00e-01 	\\ \hline
		
		\multirow{5}{*}{\begin{tabular}[c]{@{}c@{}}Pyrim\\ (74,27)\end{tabular}}
		& $L_1$	              & 2.516e-03 (1.688e-12)           &	 19.00 	      & 3.06e+00 	  \\
        & $L_{1/2}$          & 2.455e-03 (3.487e-18)           &	 19.00 	      & 8.95e+00 	\\
        & AD$_{p}$	      & 2.450e-03 (8.717e-19)           &	 19.00 	      & 1.19e+01 	\\
        & HA$_{5}\;$ (0.08)	  & 2.398e-03 (2.615e-18)           &	 \bf{15.00}   & \bf{1.03e+00} 	\\
        & HA$_{10}$ (0.08)	  & {\bf 2.309e-03} (1.743e-18)      &	 16.00 	      & 2.57e+02 	\\
        & HA$_{20}$ (0.08)	  & {\bf2.309e-03} (2.179e-18)           &	 16.00 	      & 2.33e+00 	\\
        & HA$_{30}$ (0.08)	  & {\bf2.309e-03} (3.923e-18)           &	 16.00 	      & 2.73e+00 	 \\ \hline
		
		\multirow{5}{*}{\begin{tabular}[c]{@{}c@{}}Servo\\ (167,19)\end{tabular}}
	   & $L_1$	              & 6.634e-01 (1.339e-15)           &	 16.00 	         & 3.60e+00 	\\
       & $L_{1/2}$          & 6.632e-01 (3.347e-16)           &	 18.00 	         & 9.49e+00 	\\
       & AD$_{p}$	          & 6.631e-01 (3.347e-16)           &	 16.00 	         & 2.03e+01 	\\
       & HA$_{5}\;$ (0.08)	  & {\bf6.356e-01} (1.562e-15)           &	 \bf{13.00} 	 & \bf{1.38e+00} 	\\
       & HA$_{10}$ (0.08)	  & {\bf6.356e-01} (7.811e-16)           &	 \bf{13.00} 	 & 1.93e+02 	 \\
       & HA$_{20}$ (0.08)	  & {\bf6.356e-01} (5.579e-16)      &	 \bf{13.00} 	 & 2.05e+02 	 \\
       & HA$_{30}$ (0.08)	  & 6.357e-01 (1.562e-15)           &	 \bf{13.00} 	 & 1.86e+00 	 \\ \hline

		\multirow{5}{*}{\begin{tabular}[c]{@{}c@{}}Mpg\\ (392,7)\end{tabular}}
		& $L_1$	             & 1.003e-03 (3.339e-12)             &	 7.00 	         & 1.18e+01 	 \\
        & $L_{1/2}$         & 1.003e-03 (1.743e-18)             &	 7.00 	         & 6.82e+00 	\\
        & AD$_{p}$	     & 1.003e-03 (1.308e-18)             &	 7.00 	         & 1.38e+00 	 \\
        & HA$_{5}\;$ (0.10)	 & {\bf 9.768e-04} (1.308e-18)        &	 \bf{6.00} 	     & \bf{9.05e-01} 	 \\
        & HA$_{10}$ (0.10)	 &{ \bf 9.768e-04} (1.308e-18)        &	  \bf{6.00} 	 & 1.11e+00 	\\
        & HA$_{20}$ (0.10)	 & { \bf 9.768e-04} (1.743e-18)             &	  \bf{6.00} 	 & 1.67e+00 	 \\
        & HA$_{30}$ (0.10)	 &{ \bf 9.768e-04} (2.179e-18)             &	  \bf{6.00} 	 & 2.16e+00 	\\ \hline
		\multirow{5}{*}{\begin{tabular}[c]{@{}c@{}}Space-ga\\ (3107,6)\end{tabular}}
		& $L_1$	             & 1.436e-04 (3.814e-19)             &	 6.00 	        & 1.37e+02 	 \\
        & $L_{1/2}$         & 1.430e-04 (2.179e-19)             &	 5.00 	        & \bf{2.32e+01} 	\\
        & AD$_{p}$	     & 1.455e-04 (1.090e-19)             &	 6.00 	        & 7.66e+01 	\\
        & HA$_{5}\;$ (0.05)	 & 1.421e-04 (8.172e-20)             &	 \bf{5.00} 	    & 5.73e+01 	 \\
        & HA$_{10}$ (0.05)	 & 1.421e-04 (1.634e-19)             &	 \bf{5.00} 	    & 6.47e+01 	\\
        & HA$_{20}$ (0.05)	 & 1.421e-04 (1.634e-19)             &	 \bf{5.00} 	    & 1.12e+03 	\\
        & HA$_{30}$ (0.05)	 & {\bf1.419e-04} (0.000e+00)        &	\bf{5.00} 	    & 9.42e+01 	 \\ \hline

		\multirow{5}{*}{\begin{tabular}[c]{@{}c@{}}Energy\\ (768,8)\end{tabular}}
		& $L_1$	             & 1.156e-04 (2.179e-19)             &	 7.00 	          & 2.03e+01 	 \\
        & $L_{1/2}$         & 1.149e-04 (1.090e-19)             &	 7.00 	          & 9.69e+00 	\\
        & AD$_{p}$	     & 1.156e-04 (1.498e-19)             &	 7.00 	          & 7.89e+00 	 \\
        & HA$_{5}\;$ (0.10)	 & {\bf1.141e-04} (5.448e-20)        &	\bf{5.00} 	      & \bf{3.38e+00} 	 \\
        & HA$_{10}$ (0.10)	 & {\bf1.141e-04 }(2.724e-20)        &	 \bf{5.00} 	      & 5.76e+00 	 \\
        & HA$_{20}$ (0.08)	 & {\bf1.141e-04} (1.362e-20)       &	 \bf{5.00} 	      & 3.83e+02 	 \\
        & HA$_{30}$ (0.08)	 & {\bf1.141e-04} (1.907e-19)             &	 \bf{5.00} 	      & 9.63e+00 	\\ \hline

		\multirow{5}{*}{\begin{tabular}[c]{@{}c@{}}SkillCraft1\\ (3338,18)\end{tabular}}
& $L_1$	              & 2.294e-04 (3.269e-19)      &	 8.00 	      & 1.96e+02 	 \\
& $L_{1/2}$          & 2.299e-04 (5.176e-19)      &	 7.00 	      & \bf{6.61e+00} 	 \\
& AD$_{p}$	      & 2.279e-04 (1.634e-19)      &	 8.00 	      & 1.39e+02 	\\
& HA$_{5}\;$ (0.02)	  & {\bf2.269e-04} (2.724e-20)      &	 \bf{6.00} 	  & 8.74e+02 	\\
& HA$_{10}$ (0.02)	  & {\bf2.269e-04} (2.452e-19)      &	 \bf{6.00} 	  & 8.70e+01 	\\
& HA$_{20}$ (0.02)	  & {\bf2.269e-04} (3.541e-19)      &	 \bf{6.00} 	  & 1.02e+02 	 \\
& HA$_{30}$ (0.02)	  & {\bf2.269e-04}  (2.179e-19)&	 \bf{6.00} 	  & 1.21e+02 	\\ \hline

	\end{tabular}}
\end{table}

\begin{table}[tt]
	\centering
	\caption{Average results on real datasets from {\tt randn} initialization with data fidelity term $\Phi(\cdot)=\|\cdot\|_2$.
HAFAM$_T$ consistently outperforms others across datasets.}
	\label{LR1}
	{ \scriptsize\vskip -1mm\begin{tabular}{clccc}
		\hline
		Datase                                            & \multirow{2}{*}{Algo. ($\tau$)}
		& \multirow{2}{*}{TMSE ({\tt std})}                 & \multirow{2}{*}{{\tt nnz}}     & \multirow{2}{*}{CPU}
		\\ (m,n)           &                               &                          &                  &
		\\ \hline
		\multirow{5}{*}{\begin{tabular}[c]{@{}c@{}}Diabetes\\ (422,10)\end{tabular}}
	&$L_1$	 & 1.125e-03 (3.051e-18) &	 8.00 	 & {\bf 1.84e+00} 	\\
& $L_{1/2}$	 & 1.127e-03 (1.481e-05) &	 7.11 	 & 3.22e+00 	 \\
&  AD$_{p}$	 & 1.129e-03 (2.396e-05) &	 7.22 	 & 6.59e+00 	\\
& HA$_{5}\;$ (0.15)	 & 1.155e-03 (5.948e-05) &	 6.50 	 & 2.07e+01 	\\
& HA$_{10}$ (0.15)	 & 1.135e-03 (3.637e-05) &	{\bf 5.99} 	 & 1.30e+01 	\\
& HA$_{20}$ (0.15)	 & 1.118e-03 (1.287e-05) &	 6.08 	 & 7.34e+00 	 \\
& HA$_{30}$ (0.15)	 & {\bf 1.117e-03} (1.186e-05) &	 6.06 	 & 8.73e+00 	\\ \hline

		\multirow{5}{*}{\begin{tabular}[c]{@{}c@{}}Prostate\\ cancer\\ (97,8)\end{tabular}}
		&$L_1$	 & 8.046e-01 (1.674e-15) &	 7.00 	 & 4.56e-01 	 \\
& $L_{1/2}$	 & 8.078e-01 (4.733e-02) &	 4.23 	 & 1.57e+00 	 \\
&  AD$_{p}$	 & {\bf 7.659e-01} (2.141e-08) &	 {\bf 2.00} 	 & 2.08e+00 	\\
& HA$_{5}\;$ (0.10)	 & 7.699e-01 (4.687e-03) &	 3.00 	 & {\bf 4.44e-01} 	 \\
& HA$_{10}$ (0.20)	 & 7.679e-01 (1.617e-11) &	 3.00 	 & 4.87e-01 	\\
& HA$_{20}$ (0.30)	 & 7.679e-01 (2.032e-11) &	 3.00 	 & 7.54e-01 	\\
& HA$_{30}$ (0.30)	 & 7.679e-01 (6.917e-11) &	 3.00 	 & 8.67e-01 	 \\ \hline
		\multirow{5}{*}{\begin{tabular}[c]{@{}c@{}}Pyrim\\ (74,27)\end{tabular}}
		& $L_1$	 & 2.515e-03 (1.743e-18) &	 19.00 	 & 4.83e+00 	 \\
& $L_{1/2}$	 & 2.698e-03 (2.284e-04) &	 16.87 	 & 4.92e+00 	 \\
&  AD$_{p}$	 & 2.465e-03 (2.813e-10) &	 18.00 	 & 8.31e+00 	\\
& HA$_{5}\;$ (0.05)	 & 2.550e-03 (2.543e-04) &	 {\bf 14.45} 	 & {\bf 1.31e+01} 	\\
& HA$_{10}$ (0.05)	 & 2.471e-03 (7.683e-05) &	 15.66 	 & 1.61e+01 	\\
& HA$_{20}$ (0.05)	 & {\bf 2.428e-03} (7.465e-06) &	 15.51 	 & 2.15e+01 	\\
& HA$_{30}$ (0.05)	 & {\bf 2.428e-03} (1.156e-06) &	 15.21 	 & 9.97e+00 	\\ \hline
			
		\multirow{5}{*}{\begin{tabular}[c]{@{}c@{}}Servo\\ (167,19)\end{tabular}}
	& $L_1$	 & 6.719e-01 (1.227e-15) &	 19.00 	 & 6.17e+00 	 \\
& $L_{1/2}$	 & 6.631e-01 (2.535e-04) &	 18.47 	 & 6.69e+00 	\\
&  AD$_{p}$	 & 6.631e-01 (1.139e-05) &	 17.94 	 & 3.07e+01 	 \\
& HA$_{5}\;$ (0.01)	 & 6.632e-01 (2.969e-03) &	 17.51 	 & 8.38e+00 	 \\
& HA$_{10}$ (0.01)	 & {\bf 6.630e-01} (1.448e-03) &	 {\bf 17.42} 	 & {\bf 1.00e+01} 	\\
& HA$_{20}$ (0.01)	 & 6.631e-01 (1.044e-03) &	 17.48 	 & 1.02e+01 	 \\
& HA$_{30}$ (0.01)	 & 6.632e-01 (8.674e-04) &	 17.39 	 & 9.08e+00 	 \\ \hline
		\multirow{5}{*}{\begin{tabular}[c]{@{}c@{}}Mpg\\ (392,7)\end{tabular}}
		& $L_1$	 & 1.078e-03 (1.308e-18) &	 4.00 	 & 9.52e-01 	\\
& $L_{1/2}$	 & 1.038e-03 (1.132e-04) &	 {\bf 2.95} 	 & {\bf 3.66e-01} 	 \\
&  AD$_{p}$	 & 1.022e-03 (8.388e-11) &	 4.00 	 & 1.93e+00 	\\
& HA$_{5}\;$ (0.05)	 & 1.100e-03 (2.227e-04) &	 3.38 	 & 5.60e+01 	\\
& HA$_{10}$ (0.05)	 & 1.139e-03 (2.817e-04) &	 3.00 	 & 3.73e+01 	\\
& HA$_{20}$ (0.05)	 & {\bf 1.020e-03} (9.156e-07) &	 3.40 	 & 1.14e+00 	 \\
& HA$_{30}$ (0.05)	 & 1.022e-03 (3.663e-07) &	 3.96 	 & 1.47e+00 	 \\ \hline
		\multirow{5}{*}{\begin{tabular}[c]{@{}c@{}}Space-ga\\ (3107,6)\end{tabular}}
	& $L_1$	 & 1.865e-04 (1.907e-19) &	 {\bf 3.00} 	 & 1.72e+00 	 \\
& $L_{1/2}$	 & 1.453e-04 (2.997e-07) &	 6.00 	 & 1.88e+01 	 \\
&  AD$_{p}$	 & 1.448e-04 (1.619e-13) &	 6.00 	 & 3.71e+01 	 \\
& HA$_{5}\;$ (0.01)	 & 1.454e-04 (3.518e-06) &	 5.73 	 & {\bf 8.28e-01} 	\\
& HA$_{10}$ (0.01)	 & {\bf 1.446e-04} (2.393e-07) &	 5.63 	 & 9.34e-01 	\\
& HA$_{20}$ (0.01)	 & 1.447e-04 (2.225e-07) &	 5.72 	 & 1.35e+00 	 \\
& HA$_{30}$ (0.01)	 & 1.448e-04 (3.768e-15) &	 6.00 	 & 1.65e+00 	\\ \hline
\multirow{5}{*}{\begin{tabular}[c]{@{}c@{}}Energy\\ (768,8)\end{tabular}}
	& $L_1$	 & 1.277e-04 (1.907e-19) &	{\bf 4.00} 	 & {\bf 6.98e-01} 	 \\
& $L_{1/2}$	 & 1.240e-04 (4.282e-09) &	{\bf 4.00} 	 & 1.04e+00 	\\
&  AD$_{p}$	 & 1.159e-04 (3.882e-08) &	 6.83 	 & 3.04e+01 	\\
& HA$_{5}\;$ (0.12)	 & 1.210e-04 (1.571e-05) &	 4.92 	 & 3.50e+01 	\\
& HA$_{10}$ (0.12)	 & 1.200e-04 (3.989e-06) &	 4.16 	 & 1.86e+01 	 \\
& HA$_{20}$ (0.12)	 & 1.179e-04 (3.818e-06) &	 {\bf 4.00} 	 & 1.29e+00 	\\
& HA$_{30}$ (0.12)	 & {\bf 1.155e-04} (8.291e-07) &	{\bf 4.00} 	 & 1.56e+00 	\\ \hline
		
		\multirow{5}{*}{\begin{tabular}[c]{@{}c@{}}SkillCraft1\\ (3338,18)\end{tabular}}
		& $L_1$	 & 2.293e-04 (2.724e-20) &	 8.00 	 & 2.74e+00 	 \\
& $L_{1/2}$	 & 2.297e-04 (5.408e-07) &	 7.19 	 & 3.55e+00 	 \\
&  AD$_{p}$	 & 2.278e-04 (4.941e-13) &	 8.00 	 & 5.74e+00 	 \\
& HA$_{5}\;$ (0.05)	 & 2.338e-04 (2.319e-05) &	 8.64 	 & 4.82e+00 	\\
& HA$_{10}$ (0.05)	 & 2.283e-04 (2.883e-06) &	 7.21 	 & {\bf 2.51e+00 }	 \\
& HA$_{20}$ (0.05)	 & 2.283e-04 (2.487e-06) &	 6.65 	 & 3.16e+00 	 \\
& HA$_{30}$ (0.05)	 & {\bf 2.275e-04} (1.230e-06) &	{\bf 6.43} 	 & 3.54e+00 	\\ \hline

	\end{tabular}}
\end{table}

\begin{table}[tt]
	\centering
\caption{Average results on real datasets from zero initialization with data fidelity term $\Phi(\cdot)=\|\cdot\|_2$.
 HAFAM$_{T}$ continues to perform well.}
	\label{LR2}
{\scriptsize\vskip -1mm\begin{tabular}{clccc}
		\hline
		Datase                                           & \multirow{2}{*}{Algo. ($\tau$)}
		& \multirow{2}{*}{TMSE ({\tt std})}                 & \multirow{2}{*}{{\tt nnz}}     & \multirow{2}{*}{CPU}
		\\ (m,n)           &                             &                          &                  &
		\\ \hline
		\multirow{5}{*}{\begin{tabular}[c]{@{}c@{}}Diabetes\\ (422,10)\end{tabular}}
	& $L_1$	 & 1.125e-03 (3.051e-18) &	 8.00 	 & 1.72e+00 	\\
& $L_{1/2}$	 & 1.129e-03 (1.526e-18) &	 7.00 	 & 1.05e+00 	\\
&  AD$_{p}$	 & 1.125e-03 (1.526e-18) &	 7.00 	 & 3.35e+00 	\\
& HA$_{5}\;$ (0.15)	 & {\bf 1.122e-03} (1.526e-18) &	 {\bf 6.00} 	 & {\bf 4.51e-01} \\	
& HA$_{10}$ (0.15)	 & 1.132e-03 (1.526e-18) &	 {\bf 6.00} 	 & 5.17e-01 	\\
& HA$_{20}$ (0.15)	 & {\bf 1.122e-03} (8.717e-19) &	 {\bf 6.00} 	 & 1.32e+00 \\	
& HA$_{30}$ (0.15)	 & {\bf 1.122e-03} (2.833e-18) &	  {\bf 6.00} 	 & 2.49e+00 \\	 \hline
		\multirow{5}{*}{\begin{tabular}[c]{@{}c@{}}Prostate \\ cancer\\ (97,8)\end{tabular}}
	& $L_1$	 & 8.046e-01 (1.674e-15) &	 7.00 	 & 4.20e-01 	\\
& $L_{1/2}$	 & 8.407e-01 (3.347e-16) &	 5.00 	 & {\bf 1.35e+00} 	 \\
&  AD$_{p}$	 & {\bf 7.659e-01} (0.000e+00) &	{\bf 2.00} 	 & 2.13e+00 	 \\
& HA$_{5}\;$ (0.10)	 & 7.679e-01 (6.695e-16) &	 3.00 	 & 4.45e-01 	  \\
& HA$_{10}$ (0.20)	 & 7.679e-01 (1.227e-15) &	 3.00 	 & 4.94e-01 	 \\
& HA$_{20}$ (0.30)	 & 7.679e-01 (1.562e-15) &	 3.00 	 & 7.82e-01 	\\
& HA$_{30}$ (0.30)	 & 7.679e-01 (1.116e-16) &	 3.00 	 & 8.80e-01 	\\ \hline
		
		\multirow{5}{*}{\begin{tabular}[c]{@{}c@{}}Pyrim\\ (74,27)\end{tabular}}
		& $L_1$	 & 2.520e-03 (3.487e-18) &	 21.00 	 & 3.43e+00 	\\
& $L_{1/2}$	 & 2.550e-03 (2.179e-18) &	 17.00 	 & 4.96e+00 	 \\
&  AD$_{p}$	 & 2.465e-03 (4.359e-18) &	 18.00 	 & 8.01e+00 	 \\
& HA$_{5}\;$ (0.05)	 & 2.870e-03 (7.410e-18) &	 {\bf 13.00 }	 & 1.99e+00 	 \\
& HA$_{10}$ (0.05)	 & {\bf 2.429e-03} (1.743e-18) &	 15.00 	 & {\bf 1.40e+00} 	 \\
& HA$_{20}$ (0.05)	 & {\bf 2.429e-03} (1.308e-18) &	 15.00 	 & 1.52e+00 	\\
& HA$_{30}$ (0.05)	 & {\bf 2.429e-03} (6.538e-18) &	 15.00 	 & 1.59e+00 	\\ \hline
		\multirow{5}{*}{\begin{tabular}[c]{@{}c@{}}Servo\\ (167,19)\end{tabular}}
		& $L_1$	 & 6.653e-01 (1.562e-15) &	 19.00 	 & 6.40e+00 	 \\
& $L_{1/2}$	 & 6.630e-01 (1.116e-15) &	 18.00 	 & 6.79e+00 	\\
&  AD$_{p}$	 & 6.633e-01 (1.116e-16) &	 18.00 	 & 3.09e+01 	 \\
& HA$_{5}\;$ (0.05)	 & 6.622e-01 (3.347e-16) &	 15.00 	 & 8.34e+00 	\\
& HA$_{10}$ (0.05)	 & 6.736e-01 (3.347e-16) &	 {\bf 13.00} 	 & {\bf 8.41e-01} 	\\
& HA$_{20}$ (0.08)	 & 6.640e-01 (4.463e-16) &	{\bf 13.00 }	 & 9.45e-01 	\\
& HA$_{30}$ (0.08)	 & {\bf 6.529e-01} (2.232e-16) &	{\bf 13.00 }	 & 9.42e+00 	\\ \hline

		\multirow{5}{*}{\begin{tabular}[c]{@{}c@{}}Mpg\\ (392,7)\end{tabular}}
		& $L_1$	 & 1.078e-03 (1.308e-18) &	 4.00 	 & 8.90e-01 	 \\
& $L_{1/2}$	 & 1.057e-03 (1.743e-18) &	 3.00 	 & 3.37e-01 	 \\
&  AD$_{p}$	 & 1.022e-03 (1.090e-18) &	 4.00 	 & 1.91e+00 	\\
& HA$_{5}\;$ (0.05)	 & 1.082e-03 (1.743e-18) &	 {\bf 2.00} 	 & {\bf 2.92e-01 }	\\
& HA$_{10}$ (0.05)	 & {\bf 1.020e-03} (4.359e-19) &	 3.00 	 & 2.11e+01 	 \\
& HA$_{20}$ (0.05)	 & {\bf 1.020e-03} (1.090e-18) &	 3.00 	 & 7.16e-01 	\\
& HA$_{30}$ (0.05)	 & 1.022e-03 (2.397e-18) &	 4.00 	 & 1.11e+00 	\\ \hline
		\multirow{5}{*}{\begin{tabular}[c]{@{}c@{}}Space-ga\\ (3107,6)\end{tabular}}
	& $L_1$	 & 1.865e-04 (1.907e-19) &	 {\bf 3.00} 	 & 1.65e+00 	\\
& $L_{1/2}$	 & 1.454e-04 (2.179e-19) &	 6.00 	 & 1.94e+01 	\\
&  AD$_{p}$	 & 1.448e-04 (2.179e-19) &	 6.00 	 & 3.95e+01 	\\
& HA$_{5}\;$ (0.01)	 & 1.640e-04 (2.724e-19) &	 5.00 	 & 9.69e-01 	\\
& HA$_{10}$ (0.01)	 & {\bf 1.443e-04} (2.997e-19) &	 5.00 	 &{\bf 9.67e-01} 	\\
& HA$_{20}$ (0.01)	 & 1.448e-04 (2.997e-19) &	 6.00 	 & 1.98e+00 	\\
& HA$_{30}$ (0.01)	 & 1.448e-04 (8.172e-20) &	 6.00 	 & 1.83e+00 	\\ \hline
		\multirow{5}{*}{\begin{tabular}[c]{@{}c@{}}Energy\\ (768,8)\end{tabular}}
		& $L_1$	 & 1.277e-04 (1.907e-19) &	 {\bf 4.00} 	 & 6.87e-01 	\\
& $L_{1/2}$	 & 1.240e-04 (1.907e-19) &	{\bf 4.00 }	 & 1.39e+00 	 \\
& AD$_{p}$	 & 1.158e-04 (1.090e-19) &	 6.00 	 & 2.82e+01 	\\
& HA$_{5}\;$ (0.12)	 & 1.166e-04 (1.090e-19) &	{\bf 4.00} 	 & {\bf 6.77e-01 }	\\
& HA$_{10}$ (0.12)	 & 1.166e-04 (1.090e-19) &	 {\bf 4.00 }	 & 7.18e-01 	\\
& HA$_{20}$ (0.12)	 & {\bf 1.154e-04} (1.907e-19) &	{\bf 4.00 }	 & 8.46e-01 	\\
& HA$_{30}$ (0.12)	 & {\bf 1.154e-04} (9.535e-20) &	{\bf 4.00 }	 & 1.84e+00 	\\ \hline
		
		\multirow{5}{*}{\begin{tabular}[c]{@{}c@{}}SkillCraft1\\ (3338,18)\end{tabular}}
		& $L_1$	 & 2.293e-04 (2.724e-20) &	 8.00 	 & 2.42e+00 	\\
& $L_{1/2}$	 & 2.296e-04 (4.903e-19) &	 8.00 	 & 2.53e+00 	 \\
&  AD$_{p}$	 & 2.278e-04 (2.452e-19) &	 8.00 	 & 3.61e+00 	 \\
& HA$_{5}\;$ (0.05)	 & {\bf 2.276e-04} (1.907e-19) &	 {\bf 6.00} 	 & {\bf 8.35e-01} 	 \\
& HA$_{10}$ (0.05)	 & 2.278e-04 (4.359e-19) &	 {\bf 6.00} 	 & 9.77e-01 	\\
& HA$_{20}$ (0.05)	 & 2.278e-04 (4.359e-19) &	 {\bf 6.00 }	 & 1.15e+00 	 \\
& HA$_{30}$ (0.05)	 & {\bf 2.276e-04} (1.634e-19) &	{\bf 6.00} 	 & 1.41e+00 	\\ \hline
	\end{tabular}}
\end{table}

\section{Conclusion}

We  delve into the current popular topic of $L_1$ over $L_2$ minimization,
 verifying its partial smoothness and prox-regularity.
Based on this, we validate ADMM$_p$ \cite{Tao20} (or ADMM$_p^+$ \cite{TaoZhang23}) with the property of identifying the active manifold in finite iterations. {\it It highlights that these $L_1/L_2$ proximal-friendly approaches preserve
 the essential ``support detection" property.}
 It provides a better understanding of the landscape of $L_1$ over $L_2$ minimization.

Furthermore, we propose a heuristic acceleration framework. It consists of two phases: first, we employ ADMM$_p$ or ADMM$_p^+$
 to identify the active manifold, then apply a globalized semismooth Newton method within this manifold.
Theoretically, we show that it achieves a superlinear/quadratic convergence rate under certain conditions.
Extensive numerical experiments for sparse recovery on synthetic and real datasets
 demonstrate its superiority compared to the existing state-of-the-art methods.



\section*{Acknowledgments}
 The first author extends  gratitude to Prof. Radu Ioan Bo\c{t} from the University of Vienna for his valuable and constructive comments on the manuscript. The initial version  was conducted during she visited  him in 2023.

  \begin{appendices}
      \section{Proof of Lemma \ref{proxreg}  }\label{AppA}

\begin{proof}First, invoking \cite[Lemma 2.1(ii)]{BRDL22} and $\partial(g({\bar{\h x}}) f({\bar{\h x}}))=g({\bar{\h x}})\partial f({\bar{\h x}})$ (due to $f$  prox-regular at ${\bar{\h x}}$ and $g({\bar{\h x}})> 0$),
 \begin{eqnarray*}\partial \left(\frac{f}{g}\right)({\bar{\h x}})=\frac{g({\bar{\h x}})\partial f({\bar{\h x}})-f({\bar{\h x}})\nabla g({\bar{\h x}})}{g({\bar{\h x}})^2}.
 \end{eqnarray*}
 We only need to prove for any $\bar{\h v}\in\partial f(\bar{\h x})$, there exists $\rho>0$ such that
  \begin{eqnarray}\label{proxfr}&&\frac{f({\h y})}{g({\h y})}-\frac{f({\h x})}{g({\h x})}-\left \langle \frac{{\h v} g({\h x})-f({\h x})\nabla g({\h x})}{g({\h x})^2}, {\h y}-{\h x}\right\rangle \nn\\
  &&\ge -\frac{\rho}{2}\|{\h x}-{\h y}\|_2^2,\end{eqnarray}
 whenever ${\h x}$ and ${\h y}$ are near $\bar{\h x}$ with $f({\h x})/g({\h x})$ near $f({\bar{\h x}})/g({\bar{\h x}})$, and
 $\displaystyle{\frac{{\h v} g({\h x})-f({\h x})\nabla g({\h x})}{g({\h x})^2}}$ near $\displaystyle{\frac{{\bar{\h v}} g({\bar{\h x}})-f({\bar{\h x}})\nabla g({\bar{\h x}})}{g({\bar{\h x}})^2}}$ (i.e., ${\h v}\in\partial f({\h x})$ is near ${\bar{\h v}}$).
 Below, we show that
 \begin{eqnarray*}&&\left|\frac{f({\h x})}{g({\bar{\h x}})}-\frac{f({\bar{\h x}})}{g({\bar{\h x}})}\right |
 \le \left| \frac{f({\h x})}{g({\bar{\h x}})}-\frac{f({{\h x}})}{g({{\h x}})}\right|+\left|\frac{f({\h x})}{g({{\h x}})}-\frac{f({\bar{\h x}})}{g({\bar{\h x}})}\right|\nn\\
 &&\le\frac{|f({\h x})(g({\h x})-g({\bar{\h x}}))|}{g({\bar{\h x}})g({\h x})}+\left|\frac{f({\h x})}{g({{\h x}})}-\frac{f({\bar{\h x}})}{g({\bar{\h x}})}\right|. \end{eqnarray*}
 Thus, $f({\h x})$ near $f({\bar{\h x}})$ whenever $f({\h x})/g({\h x})$ near $f({\bar{\h x}})/g({\bar{\h x}})$ and ${\h x}$ near ${\bar{\h x}}$.
 Since $f$ is prox-regular at the point ${\bar{\h x}}$, then
 for the ${\bar{\h v}}\in\partial f({\bar{\h x}})$, there exists $\rho_1>0$ such that
 $$  f({\h y})-f({\h x}) - \langle {\h v}, {\h y}-{\h x} \rangle \ge -\frac{\rho_1}{2}\|{\h x}-{\h y}\|_2^2,$$
whenever ${\h x}$ and ${\h y}$ are near ${\bar{\h x}}$ with $f({\h x})$ near $f({\bar{\h x}})$, and ${\h v}\in \partial f({\h x})$ is near ${\bar{\h v}}$  where ${\bar{\h v}}\in\partial f({\bar{\h x}})$. In view of $g({\bar{\h x}})>0$,
there exists two constants $m_1$ and $M_1$ such that
 \begin{eqnarray}\label{lug} 0<m_1\le g({\h x})\le M_1,\end{eqnarray}
 when ${\h x}$ is near ${\bar {\h x}}$.
 Furthermore,
\begin{eqnarray}\label{eq1}\frac{1}{g({\h x})}\left[f({\h y})-f({\h x})-\langle {\h v},{\h y}-{\h x} \rangle \right]\ge -\frac{{\tilde \rho}_1}{2}\|{\h x}-{\h y}\|_2^2,\end{eqnarray}
 where ${\tilde \rho}_1 = \rho_1/m_1.$
 To proceed, we
 verify the following inequality:
 \begin{eqnarray}\label{eq2}&&\frac{f({\h y})}{g({\h y})}\frac{1}{g({\h x})}\left( g({\h x})-g({\h y})\right)+\frac{f({\h x})}{g({\h x})^2} \langle \nabla g({\h x}),{\h y}-{\h x} \rangle\nn\\
 &&\ge -\frac{\rho_2}{2}\|{\h x}-{\h y}\|_2^2.\end{eqnarray}
 We divide into two cases to prove:  (a) $f({\h x})=0$; (b) $f({\h x})\neq0$.\\
Case (a): $f({\h x})=0$.
 If $f({\h y})=0$, the above inequality holds obviously. If $f({\h y})\neq {\bf 0}$,
 Because $f$ is locally Lipschitz continuous ${\bar{\h x}}$, there exist two constants $m_2$ and $M_2$ such that $m_2\le f({\h y})\le M_2$
 ($M_2$, $m_2$ can take the same sign as  $f({\h y})$) when
 ${\h y}$ is near ${\h {\bar x}}$. Without loss of generality, we assume that $m_2>0$.
 By using the locally Lipschitz continuity of $g$ near ${\bar{\h x}}$ (with constant $L_g$) and combining with (\ref{lug}),
 there exists a constant $\rho_2$ such that (\ref{eq2}) holds.\\
Case (b): $f({\h x})\neq0$. We first show that $H({\h x})=\frac{f({\h x})}{g({\h x})}$ is locally Lipschitz continuous around ${\bar{\h x}}$.
Let $L_H=\frac{M_2}{m_1^2} L_g+\frac{L_f}{m_1}$.
\begin{eqnarray*}&& |H({\h x})-H({\h y})| \le\frac{H({\h x})}{g({\h y})}|g({\h y})-g({\h x})|+\frac{1}{g(\h y)}|f({\h x})-f({\h y})|\nn\\
&& \le  L_H\|{\h x}-{\h y}\|_2.\end{eqnarray*}
\begin{eqnarray*}\lefteqn{ \frac{f({\h y})}{g({\h y})g({\h x})}\left( g({\h x})-g({\h y})\right)+\frac{f({\h x})}{g({\h x})^2} \langle \nabla g({\h x}),{\h y}-{\h x} \rangle}\nn\\
&= &\frac{1}{g({\h x})}\left( H({\h y})-H({\h x})\right)(g({\h x})-g({\h y}) )\nn\\
&+&\frac{f({\h x})}{g({\h x})^2}\left(g({\h x})-g({\h y}) +\langle \nabla g({\h x}),{\h y}-{\h x} \rangle \right)\nn\\
&\ge& -\frac{1}{m_1}L_HL_g\|{\h x}-{\h y}\|_2^2 -\frac{L_{\nabla g}}{2}\frac{M_2}{m_1^2}\|{\h x}-{\h y}\|_2^2,
\end{eqnarray*}
where  $L_{\nabla g}$ is the locally Lipschitz constants of $\nabla g$.
Thus, the inequality (\ref{eq2}) is valid with $\rho_2=\frac{2L_H L_g}{m_1}+\frac{M_2 L_{\nabla g}}{m_1^2}$.

By multiplying (\ref{eq1}) on both sides with $\frac{g({\h y})}{g({\h y})}$ and adding with (\ref{eq2}), it leads to
(\ref{proxfr}) is valid with $\rho={\tilde \rho}_1+\rho_2$.
Therefore, the function $\displaystyle{\frac{f}{g}}$ is prox-regular at the point
${\bar{\h x}}$.
 \end{proof}

\section{Proof of Proposition \ref{PhiPS}  }\label{AppB}

\begin{proof} First, we show that
 $h(\cdot)$ defined in (\ref{hfun})  is partly smooth at ${\h x}_0$ relative to ${\cal M}_{{\h x}_0}$.
We divide into two cases to verify: (a) ${\cal X}={\mathbb R}^n$; (b) ${\cal X}={\mathbb R}_+^n$.\\[0.1cm]
Case (a). By invoking Definition \ref{def2.3}, it is evident that properties
(i) and (iv) hold. Additionally, we can refer to Equation (2.3) in \cite{ZengYuPong20} to show the validity of property (ii).
We only need to verify the property (iii) normal sharpness, i.e.,
\begin{eqnarray*} dh({\h x}_0)(-{\h w})>-dh({\h x}_0)({\h w})\end{eqnarray*}
for all nonzero directions ${\h w}$ in $N_{{\cal M}_{\h x_0}}({\h x})$.
Note that
\begin{eqnarray*} h({\h x}_0+\tau {\h w})=\frac{\|{\h x}_0+\tau{\h w}\|_1}{\|{\h x}_0+\tau{\h w}\|_2}=\frac{\|{\h x}_0\|_1+\tau\|{\h w}\|_1}{\sqrt{\|{\h x}_0\|_2^2+\tau^2\|{\h w}\|_2^2}}\end{eqnarray*}
when ${\h w}$ in $N_{{\cal M}_{\h x_0}}({\h x})$. Recall $h(\cdot)$ is defined in (\ref{hfun}).
By some routine calculations, it yields that
\begin{eqnarray*}dh({\h x}_0)({\h w})= \lim_{\tau\downarrow 0}\inf_{{\bar{\h w}}\to{\h w}}\frac{h({\h x}_0+\tau{\bar{\h w}})-h({\h x}_0)}{\tau} =\frac{\|{\h w}\|_1}{\|{\h x}_0\|_2},\end{eqnarray*}
and
\begin{eqnarray*} dh({\h x}_0)(-{\h w})= \lim_{\tau\downarrow 0}\inf_{{\bar{\h w}}\to{\h w}}\frac{h({\h x}_0-\tau{\bar{\h w}})-h({\h x}_0)}{\tau} =\frac{\|{\h w}\|_1}{\|{\h x}_0\|_2}. \end{eqnarray*}
Thus,
the property (iii) of Definition \ref{def2.3} holds for the $h(\cdot)$ defined in (\ref{hfun}).

\noindent Case (b). The proof for property (i)-(iii) are similar to Case (a).
For (iv), we only need to show that ${\partial h}$ is inner semicontinuous at ${\h x}_0$ relative ${\cal M}_{{\h x}_0}$.
First note that
\begin{eqnarray*} h({\h x})=h_1(\h x)+\iota_{{\mathbb R}^n_+}(\h x),\end{eqnarray*}
where $h_1(\h x):=\frac{{\h e}^\top {\h x}}{\|\h x\|_2}$.
According to \cite[Corollary 10.9 and Exercise 10.10]{RockWetsVA}, we have
\begin{eqnarray} \label{sumrule} \partial h({\h x})=\partial h_1 (\h x) +N_{{\mathbb R}^n_+}(\h x),\end{eqnarray}
where $N_{{\mathbb R}^n_+}(\h x)$ is the limiting normal cone of ${{\mathbb R}^n_+}$ at $\h x$.
\begin{eqnarray}\label{normcone}
N_{{\mathbb R}^n_+}(\h x) = \{{\h d}\in \mathbb R^n \; |\; {\h d}_{\Lambda} ={\bf 0},\; {\h d}_{{\Lambda}^c} \le {\bf 0},\; \Lambda={\text{supp}}(\h x)\},
\end{eqnarray}
and invoking \cite[Lemma 2.1(ii)]{BRDL22},
$\partial h_1(\h x) = \frac{{\bf e}}{r} - \frac{a}{r^3} {\h x},$
where $a=\|{\h x}\|_1$ and $r=\|{\h x}\|_2$.

 Next, without loss of generality, we assume that
  $({\h x}_0)^\top=\left( ({\h x}_{1})^\top, ({\h x}_{2})^\top\right)$ where
  ${\h x}_1= ({\h x}_0)_{\Lambda}$ ($\Lambda={\text{supp}}({\h x}_0)$) and ${\h x}_2= {\bf 0}$.

  For any ${\h y}\in\partial  h({\h x}_0)$,
 it follows from (\ref{sumrule}) and  (\ref{normcone}) that
\begin{eqnarray} \label{vecy} {\h y}=\left(\frac{{\bf e}}{r}-\frac{a}{r^3} {\h x}_0\right)+{\h d}\end{eqnarray}
  where $a=\|{\h x}_0\|_1$ and $r=\|{\h x}_0\|_2$ and
  $({\h d})^\top=\left(({\h d}_1)^\top,({\h d}_2)^\top\right)$ (where ${\h d}_1$ and ${\h d}_2$ are in the same dimension
   as ${\h x}_1$ and ${\h x}_2$) and ${\h d}_1={\bf 0}\in{\mathbb R}^{\sharp(\Lambda)}$ and ${\h d}_2\le {\bf 0}$.

  For any sequence of points ${\h x}^k$ in ${\cal M}_{{\h x}_0}$ approaching ${\h x}_0$, we take ${\h y}^k=\left(\frac{{\bf e}}{r_k}-\frac{a_k}{({r_k})^3} {\h x}^k\right)+{\h d}$ ($a_k=\|{\h x}^k\|_1$, $r_k=\|{\h x}^k\|_2$) where ${\h d}$ is given in (\ref{vecy}).
 We have ${\h y}^k\in\partial h({\h x}^k)$ due to ${\h x}^k\in{\cal M}_{{\h x}_0}$ and ${\h y}^k\to {\h y}$.
  Thus, ${\partial h}$ is inner semicontinuous at ${\h x}_0$ relative ${\cal M}_{{\h x}_0}$.
So, $h$ is partly smooth at any ${\h x}_0$ relative to ${\cal M}_{{\h x}_0}$.

   For both cases, we show that $h(\cdot)$ defined in (\ref{hfun})  is partly smooth at ${\h x}_0$ relative to ${\cal M}_{{\h x}_0}$.
  Invoking \cite[Corollary 4.7]{Lewis02}, $F_u(\cdot)+\iota_{\cal X}(\cdot)$ is partly smooth at ${\h x}_0$ relative to ${\cal M}_{{\h x}_0}$.
\end{proof}

\section{Proof of Theorem  \ref{thmidentify}}\label{AppC}
\begin{proof}First, we have ${\h x}^\infty\neq{\bf 0}$ due to ${A^\top {\h b}}\not\in {\cal X}^o$.
For the comprehensive proof of this argument, we refer the reader to the detailed proof of \cite[Theorem 5.8]{Tao20} and \cite[Theorem 8]{TaoZhang23}.
For the model (\ref{L1o2uncon}) with the general $\Phi(\h x)$ ($L$-smooth),
the global convergence of (\ref{ADMMschI})
 holds  according to Theorem \ref{thmglobal}.
For either ${\cal X}={{\mathbb R}^n_+}$ or ${\cal X}={\mathbb R}^n$, according to  \cite[Corollary 10.9]{RockWetsVA},  the following relation holds
\begin{eqnarray} \label{calsumrule} \partial\left(\frac{\|\cdot\|_1}{\|\cdot\|_2}+\iota_{\cal X}(\cdot)\right)({\h x}) =\partial\frac{\|{\h x}\|_1}{\|{\h x}\|_2} + \partial \iota_{\cal X}({\h x}).\end{eqnarray}
This is a consequence of the regularity of both $\frac{\|\cdot\|_1}{\|\cdot\|2}$ and $\iota_{\cal X}(\cdot)$.
By invoking the optimality condition of (\ref{ADMMschI}) and combining (\ref{calsumrule}), we have that
\begin{eqnarray*}
\left\{\begin{array}{l}
{\bf 0}\!\in\!\gamma\left(\displaystyle{\frac{\sign(\h x^{k+1})}{r_{k+1}}}\!-\!\frac{a_{k+1}}{(r_{k+1})^3}{\h x}^{k+1}+\partial \iota_{\cal X}(\h x^{k+1})\right)\\[0.25cm]
\;\;+\beta\left({\h x}^{k+1}-{\h y}^k\!+\!\displaystyle{\frac{{\h z}^k}{\beta}}\right)\\[0.25cm]
\nabla \Phi({\h y}^{k+1})+\beta({\h y}^{k+1}-{\h x}^{k+1}-\frac{1}{\beta}{\h z}^k) =0\\[0.25cm]
{\h z}^{k+1} = {\h z}^k+\beta ({\h x}^{k+1}-{\h y}^{k+1}),
\end{array}\right.
\end{eqnarray*}
where $a_{k+1}=\|{\h x}^{k+1}\|_1$ and $r_{k+1}=\|{\h x}^{k+1}\|_2$.
By substituting the third into the second, we have $\nabla \Phi({\h y}^{k+1})={\h z}^{k+1}$.
Furthermore,  by invoking (\ref{calsumrule}), we arrive at
\begin{eqnarray*}
&&\dist({\partial (F_u(\cdot)+ \iota_{\cal X}(\cdot))({\h x}^{k+1}),{\bf 0}}) \nn\\
&&=\dist\left[\gamma\left(\frac{\sign(\h x^{k+1})}{r_{k+1}}-\frac{a_{k+1}}{(r_{k+1})^3}{\h x}^{k+1}+\partial \iota_{\cal X}(\h x^{k+1})\right)\right.\nn\\
&&\;\;\left.+\nabla \Phi({\h x}^{k+1}),{\bf 0}\right]\nn\\
&&\le\|-\beta({\h x}^{k+1} -{\h y}^k+\frac{1}{\beta}{\h z}^k) +\nabla \Phi({\h x}^{k+1})\|_2\nn\\
&&=\| -{\h z}^{k+1}+{\beta}({\h y}^k-{\h y}^{k+1})+\nabla \Phi({\h x}^{k+1})\|_2\nn\\
&&=\| -\nabla\Phi({\h y}^{k+1}) +\nabla\Phi({\h x}^{k+1}) +\beta({\h y}^k-{\h y}^{k+1})\|_2\nn\\
&&\le L\|{\h x}^{k+1}-{\h y}^{k+1}\|_2+\beta\|{\h y}^k-{\h y}^{k+1}\|_2\nn\\
&&\le \frac{L}{\beta}\|{\h z}^k-{\h z}^{k+1}\|_2+\beta\|{\h y}^k-{\h y}^{k+1}\|_2.
\end{eqnarray*}
On the other hand, it follows from $\nabla \Phi({\h y}^{k+1})={\h z}^{k+1}$
that
\begin{eqnarray*}\|{\h z}^{k+1}-{\h z}^k\|_2=\|\nabla \Phi({\h y}^{k+1})-\nabla \Phi({\h y}^{k})\|_2 \le L\|{\h y}^{k+1}-{\h y}^k\|_2.\end{eqnarray*}
By combining the last two inequalities, we have
\begin{eqnarray*}
&&{\dist(\partial (F_u(\cdot)+ \iota_{\cal X}(\cdot))({\h x}^{k+1}),{\bf 0})}\nn\\
&&\le \left(\frac{L^2}{\beta}+\beta\right)\|{\h y}^k -{\h y}^{k+1}\|_2\rightarrow 0,
\end{eqnarray*}
as $k\to \infty$.

In addition, the function $F_u({\h x})+ \iota_{\cal X}({\h x})$ is partly smooth at the point
${{\h x}^\infty}$ relative to the manifold ${\cal M}_{{\h x}^\infty}$ (as shown in Proposition \ref{PhiPS}).
Next, we show it is prox-regular at ${\h x}^\infty$.
From Definition \ref{defprox}, we see that $f({\cdot})=\|\cdot\|_1+\iota_{\cal X}(\cdot)$ is prox-regular ${\h x}^\infty$ for ${\cal X}={\mathbb R}^n$ or ${\cal X}={\mathbb R}_+^n$.
Thus, by setting $f(\cdot)=\|\cdot\|_1+\iota_{\cal X}(\cdot)$ and
$g(\cdot)=\|\cdot\|_2$ in Lemma \ref{proxreg}, we have that
$h(\cdot)$ defined in (\ref{hfun}) is prox-regular at ${\h x}^\infty$.

Invoking \cite[Lemma 2.1(i)]{BRDL22}, we have that
\begin{eqnarray*}\partial(h+\Phi)({\h x})=\partial h({\h x})+\nabla\Phi({\h x}),\end{eqnarray*}
due to $h$ and $\Phi$ are regular.
Noting that for any ${\h v}\in\partial(F_u+ \iota_{\cal X})({\h x}^\infty)$, we have
${\h v}={\h v}_1+{\h v}_2$ where ${\h v}_1\in \partial h({\h x}^\infty)$ and
${\h v}_2\in \nabla \Phi({\h x}^\infty)$.
By using $h(\cdot)$ and $\Phi(\cdot)$ are prox-regular at ${\h x}^\infty$ and recall
Definition \ref{defprox},
we can obtain that $F_u(\cdot)+ \iota_{\cal X}(\cdot)$ is prox-regular
at ${\h x}^\infty$.
According to Theorem \ref{identify}, we have that for all large $k$,
${\h x}_k\in{\cal M}_{{\h x}^\infty}.$
\end{proof}

\section{Proof of Theorem \ref{sublim}}\label{AppD}
 \begin{proof}
The assertion (i) is valid obviously.
For (ii), the direction ${\h u}^j$ is always a decent direction of $\varphi(\cdot)$ if the sequence doesn't achieve the
stationary point (i.e., $\varphi({\hat{\h u}})=0$). Since
$\varphi$ is level bounded on ${\cal M}_{{\h x}^\infty}$, the sequence of $\{{\h u}^j\}$ is bounded.
Assume that ${\hat{\h u}}$ be an accumulation point of $\{{\h u}^j\}$.
Then, there exists a subsequence $\{{\h u}^{k_j}\}$ converging to ${\hat{\h u}}$ and
$\nabla \varphi({\hat{\h u}})=0$ by using the standard analysis \cite[Theorem 6.3.3]{doi:10.1137/1.9781611971200}.

Next, invoking for all $V\in\partial^2 \varphi({\hat{\h u}})$ positive definite,
then there exists an open ball ${\cal B}({\hat{\h u}})$ centered at ${\hat{\h u}}$ such that
the matrices set of ${\cal V}=\{V\in \partial^2 \varphi({\h u})|{\h u}\in{\cal B}({\hat{\h u}})\}$ are uniformly
positive definite.
Next, we show that for $j$ sufficiently large, $V_{k_j}\in{\cal V}$, we can always find an approximate solution of
${\h d}^{k_j}$ of (\ref{newton}) with $j:=k_j$
 such that both
(\ref{vj}) with $j:=k_j$ and
\begin{eqnarray}\label{suffdes} \langle\nabla \varphi({\h u}^{k_j}),{\h d}^{k_j} \rangle\!\le\!-\nu_{k_j}\|{\h d}^{k_j}\|_2^2\end{eqnarray}
hold.  Additionally, (\ref{vj}) with $j:=k_j$  leads to
\begin{eqnarray}\label{nablavarphi}\|\nabla\varphi({\h u}^{k_j})\|_2\le \frac{1}{1-\eta}\|V_{k_j}\|_2\|{\h d}^{k_j}\|_2.\end{eqnarray}
Below, we prove in two steps: (a) there exists a subsequence $\{{\h d}^{k_j}\}$ such that (\ref{newton}), (\ref{vj}) and (\ref{suffdes}) hold.
(b) Using induction, we can see that there exists an index ${\widehat k}$ such that
for $j\ge {\widehat k}$, ${\h u}^j\in {\cal B}({\hat{\h u}})$, $\alpha_j=1$ (taking unit step size).
\begin{itemize}
\item[(a)] Since the set ${\cal V}$ is uniformly positive definite, and define $\rho=\inf_{V\in{\cal V}} \lambda_{\min}(V)$ and thus $\rho>0$.
There exists a subsequence $k_j$ such that $\nabla\varphi({\h u}^{k_j})\to 0$ and
an approximate solution ${\hat{\h d}}^{k_j}$ of (\ref{newton}) satisfying
\begin{eqnarray}\label{semismooth}&& \langle \nabla\varphi({\h u}^{k_j}),{\hat{\h d}}^{k_j} \rangle \nn\\
&&=\langle \nabla \varphi({\h u}^{k_j})+V_{k_j}{\hat {\h d}}^{k_j}, {\hat {\h d}}^{k_j} \rangle-\langle V_{k_j}{\hat {\h d}}^{k_j}, {\hat {\h d}}^{k_j}\rangle\nn\\
 &&\overset{(\ref{vj})}{\le}\eta_{k_j}\|\nabla \varphi({\h u}^{k_j})\|_2\|{\hat{\h d}}^{k_j}\|_2-\rho\|{\hat{\h d}}^{k_j}\|_2^2\nn\\
 &&\le \|\nabla \varphi({\h u}^{k_j})\|_2^2\|{\hat{\h d}}^{k_j}\|_2-\rho\|{\hat{\h d}}^{k_j}\|_2^2\nn\\
 &&\le \frac{1}{1-\eta}\|V_{k_j}\|_2\|{\hat{\h d}}^{k_j}\|_2^2\|\nabla \varphi({\h u}^{k_j})\|_2-\rho\|{\hat{\h d}}^{k_j}\|_2^2\nn\\
 &&\le -\frac{\rho}{2}\|{\hat{\h d}}^{k_j}\|_2^2,
 \end{eqnarray}
 The third-to-last and the last inequalities of  (\ref{semismooth})  are due to the definition of $\eta_{k_j}$ and
 $\nabla \varphi({\h u}^{k_j})\to 0$, respectively. The penultimate is due to (\ref{nablavarphi}).
Inequality (\ref{semismooth}) implies that ${\hat{\h d}}^{k_j}$ satisfies
 (\ref{newton}), (\ref{vj}) and (\ref{suffdes}).
It follows from \cite[Lemma 5.2]{QiSun} that for all sufficiently large $j$, $\alpha_{k_j}=1$,
 and $\h{u}^{k_j+1}=\h{u}^{k_j}+\h{d}^{k_j}.$
\item[(b)] By invoking  $\nabla \varphi$ is semismooth at ${\hat{\h u}}$ or
 $\nabla \varphi$ is  strongly semismooth at ${\hat{\h u}}$, for sufficiently large $k$, we have
 \begin{eqnarray*} \|{\h u}^k+{\h d}^k-{\hat{\h u}}\|_2=o(\|{\h u}^k-{\hat{\h u}}\|_2)\end{eqnarray*}
 or
 \begin{eqnarray*} \|{\h u}^k+{\h d}^k-{\hat{\h u}}\|_2=O(\|{\h u}^k-{\hat{\h u}}\|_2^2),\end{eqnarray*}
 respectively.
 Then, for all sufficiently large $j$, ${\h u}^{k_j}\in {\cal B}({\hat{\h u}})$,
 thus ${\h u}^{k_j+1}\in {\cal B}({\hat{\h u}})$. Then, following the above proof, we see that
 $V_{k_j+1}\in {\cal V}$ and there exists ${\h d}^{k_j+1}$ satisfying  (\ref{newton}), (\ref{vj}) and (\ref{suffdes}).
 Thus, $\alpha_{k_j+1}=1$.
  By induction, there exists an index $k_J$ such that
for $j\ge k_J$, ${\h u}^j\in {\cal B}({\hat{\h u}})$, $\alpha_j=1$ (taking unit step size).
\end{itemize}
The assertion (ii) follows directly. Clearly, assertion (iii) is valid.
\end{proof}
  \end{appendices}


\end{document}


{\setcounter{equation}{0}
\renewcommand{\theequation}{S.\arabic{equation}}

{\setcounter{equation}{0}
	\renewcommand{\theequation}{S.\arabic{equation}}

\bibliographystyle{IEEEtran}
\title{Supplementary Materials  ``On Partly Smoothness, Activity Identification and  Faster Algorithms of  $L_1$ over $L_2$  Minimization"}
\author{ Min Tao and Xiao-Ping  Zhang and Zi-Hao Xia
}

\maketitle


\begin{Def}\label{def2.3} ({\bf Partly Smooth}) \cite[Definition 2.7]{Lewis02} Suppose that the set ${\cal M}\subset{\mathbb R}^n$ contains the point $\h x$. The function $f: {\mathbb R}^n\rightarrow{\overline {\mathbb R}}$
is partly smooth at $\h x$ relative to ${\cal M}$ if ${\cal M}$ is a manifold around $\h x$ and the following four properties hold:
\begin{itemize}
\item[(i)]({\bf Restricted Smoothness}) the restriction $f|_{\cal M}$ is smooth around $\h x$;
\item[(ii)] ({\bf Regularity}) at every point close to $\h x$ in ${\cal M}$, the function $f$ is regular and has a subgradient;
\item[(iii)] ({\bf Normal Sharpness}) $df({\h x})(-{\h w})>-df({\h x})({\h w})$ for all nonzero directions $\h w$ in $N_{\cal M}({\h x})$;
\item[(iv)] ({\bf Subgradient Continuity}) the subdifferential map $\partial f$ is continuous at $\h x$ relative to ${\cal M}$.
\end{itemize}
\end{Def}

\begin{Def}\label{PSset}({\bf Partly smooth sets})  \cite[Definition 2.8]{Lewis02}
A set $S\subset{\mathbb R}^n$ is partly smooth at a point ${\h x}$ relative to a set ${\cal M}$ if
$\iota_{S}$ is partly smooth at ${\h x}$ relative to ${\cal M}$. We say $S$
is partly smooth relative to a set ${\cal M}$ if ${\cal M}$ is a manifold and $S$ is partly smooth at each point in ${\cal M}$
relative to ${\cal M}$.
\end{Def}
The ensuing two concepts stem from the realm of sensitivity analysis, with the former serving as the foundation for the latter.
\begin{Def}\label{stronglocal}({\bf strong local minimizer}) \cite[Definition 5.4]{Lewis02}
Given any subset ${\cal M}$  of ${\mathbb R}^n$,
a point ${\h x}_0$ is a strong local minimizer of a function $f:{\cal M}\mapsto\overline{\mathbb R}$
if there exists a real $\delta>0$ such that
$f({\h x})\ge f({\h x}_0) +\delta\|{\h x}-{\h x}_0\|^2$ for all ${\h x}\in{\cal M}$ near ${\h x}_0$.
\end{Def}

\begin{Def}\label{strongcrit}({\bf strong critical point}) \cite[Definition 5.6]{Lewis02}
Suppose the function $f:{\mathbb R}^n\mapsto\overline{\mathbb R}$ is partly smooth
at the point ${\h x}_0$ relative to the set ${\cal M}\subset{\cal X}$.
We call ${\h x}_0$  a strong critical point  of $f$  relative to ${\cal M}$ if
\begin{itemize}
\item[(i)]${\h x}_0$ is a strong local minimizer of $f|_{\cal M}$, and
\item[(ii)]$0\in{\text{\rm ri}} \partial f({\h x}_0)$.
\end{itemize}
\end{Def}
\begin{ass}\label{tran}For Euclidean spaces $Y$ and $Z$, the set
$Q\subset Y\times Z$ is a manifold containing the point $({\h y}_0,{\h z}_0)$ and
satisfies the condition
$$ ({\h w},{\bf 0})\in N_{Q}({\h y}_0,{\h z}_0) \Rightarrow {\h w}={\bf 0},$$
\end{ass}
\noindent where $N_{Q}({\h y},{\h z})$ is the normal space to $Q$ at $({\h y},{\h z})$.
To proceed, we introduce several notions.
For each vector ${\h y}\in Y$, we define the set
$$Q_{\h y}=\{{\h z}\in Z|({\h y}, {\h z})\in Q\}.$$
Consider a function $p: Y\times Z \rightarrow \overline{\mathbb R}$,
and we define a function $p_{\h y}:Z \rightarrow{\overline{\mathbb R}}$
by $p_{\h y}({\h z})=p({\h y},{\h z})$ for ${\h y}\in Y$ and ${\h z}\in Z$.
The subsequent theorem demonstrates that under certain conditions, strong criticality implies that the
parametrized minimizer is also a strong critical point.

\begin{theorem}\label{strongcrpa}  \cite[Theorem 5.7]{Lewis02}
Suppose Assumption S.\ref{tran} holds and
the function $p(\cdot)$ is partly smooth relative to the manifold $Q$.
If the point ${\h z}_0$ is a strong critical point of the function
$p_{{\h y}_0}(\cdot)$ relative to the set $Q_{{\h y}_0}$, then there are open neighborhoods $U\subset Z$ of ${\h z}_0$
and $V\subset Y$ of ${\h y}_0$  and a continuously differentiable function $\Psi:V\mapsto U$ satisfying
$\Psi({\h y}_0)={\h z}_0$  and for all ${\h y}\in V$:
\begin{itemize}\item[(i)] the function $p_{\h y}|_{Q_{\h y}\cap U}$ has a unique critical point $\Psi({\h y})$;
\item[(ii)] $\Psi({\h y})$ is a strong critical point of the function $p_{\h y}(\cdot)$ relative  the manifold $Q_{\h y}\cap U$.
\end{itemize}
\end{theorem}
The following lemma states that prox-regularity  exactly ensures the well-definedness
 of the projection mapping. The term ``prox-normal neighborhood" denotes the neighborhood $\mathcal{V}$ of $\cal S$ at ${\bar{\h x}}$.
\begin{lemma} \label{proj}  \cite[Lemma 2.2]{LW04} Suppose the set ${\cal S}\in{\mathbb R}^m$ is closed.
Then, ${\cal S}$ is prox-regular at the point ${\bar{\h x}}\in{\cal S}$ if and only if
the projection mapping $P_{\cal S}$ is single value near ${\bar{\h x}}$. In this case,
there exists an open neighbourhood $\cal V$ of ${\bar{\h x}}$ on which
the following properties hold:
\begin{itemize}
\item[(i)] $P_{\cal S}(\cdot)$ is single valued and Lipchitz continuous on $\cal V$.
\item[(ii)] $P_{\cal S}(\cdot)=(I+N_{\cal S})^{-1}(\cdot)$ on $\cal V$.
\item[(iii)] For any point ${\h x}$ and ${\h v}\in {\cal V}$ such that ${\h v}-{\h x}\in N_{\cal S}({\h x})$ implies
that ${\h x}=P_{\cal S}({\h v})$.
\end{itemize}
\end{lemma}

Below, we first prove the following lemma.
\begin{lemma}\label{min}
If ${\cal M}$ is prox-regular at ${\bar {\h x}}\in{\cal M}$, the normal vector ${\bar{\h y}}-{\bar{\h x}}\in N_{\cal M}({\bar{\h x}})$ is sufficiently small such that $2({\bar{\h y}}-{\bar{\h x}})$ in the ``prox-normal neighborhood" of ${\cal M}$, then
\begin{eqnarray*}  {\h x}\in{\cal M} \Rightarrow \|{\h x}-{\bar{\h y}}\|^2\ge \|{\bar{\h x}}-{\bar{\h y}}\|^2+\frac{1}{2}\|{\h x}-{\bar{\h x}}\|^2.\end{eqnarray*}
\end{lemma}
\begin{proof}Since $\|{\h x}-{\bar{\h y}}\|^2=\|{\h x}-{\bar{\h x}}\|^2+\|{\bar{\h x}}-{\bar{\h y}}\|^2+ 2\langle{\h x}-{\bar{\h x}},{\bar{\h x}}-{\bar{\h y}} \rangle$, we only need to verify that
\begin{eqnarray} \label{s2}\|{\h x}-{\bar{\h x}}+2({\bar{\h x}}-{\bar{\h y}})\|^2\ge \|2({\bar{\h x}}-{\bar{\h y}})\|^2.\end{eqnarray}
Since ${\cal M}$ is prox-regular at ${\bar {\h x}}$, we choose sufficiently small normal vector
$2({\bar{\h y}}-{\bar{\h x}})$ such that it is in the ``prox-normal neighborhood" of ${\cal M}$
at ${\bar{\h x}}$. It leads to $P_{\cal M}[{\bar{\h x}}+2({\bar{\h y}}-{\bar{\h x}})]={\bar{\h x}}$
 due to Lemma S. \ref{proj} (iii).
 Furthermore,
 \begin{eqnarray*}\|{\h x}-{\bar{\h x}}+2({\bar{\h x}}-{\bar{\h y}})\|^2 &\ge
  &\min \{ \|{\h x}- ({\bar{\h x}}+2({\bar{\h y}-{\bar{\h x}}}))\|^2 |{\h x}\in{\cal M}\}\nn\\
  &=&\| P_{\cal M}({\bar{\h x}}+2({\bar{\h y}-{\bar{\h x}}}))-({\bar{\h x}}+2({\bar{\h y}-{\bar{\h x}}})) \|^2\nn\\
  &=&\|2({\bar{\h y}}-{\bar{\h x}})\|^2.\end{eqnarray*}
  Thus, (\ref{s2}) holds. The conclusion follows directly.
    \end{proof}

The following theorem relax the condition  in \cite[Theorem 3.3]{LW04} of ``the set $\cal S$ is
${\cal C}^p$-partly smooth ($p\ge 2$)" to
``the set $\cal S$ is partly smooth (i.e., $\iota_{\cal S}$ be ${\cal C}^1$)" (see Definition S.\ref{PSset}).

\begin{theorem}\label{spj}
Let the set $\cal S$ be prox-regular and partly smooth at the point ${\bar{\h x}}$ relative to ${\cal M}$.
For any sufficiently small normal vector ${\bar{\h n}}:={\bar{\h y}}-{\bar{\h x}}\in{\text{\rm ri}}N_{\cal S} ({\bar{\h x}})$ (${\bar{\h n}}$ in prox-normal neighbourhood),
there exists a neighbourhood $V$ of $({\bar{\h x}}+{\bar{\h n}})$ on which
the projection mappings satisfy $P_{\cal S}=P_{\cal M}$ on $V$.
\end{theorem}
\begin{proof}
Define the function $\rho$
\begin{eqnarray*}
\begin{array}{rl}
\rho:{\mathbb R}^m \times {\mathbb R}^m& \mapsto \overline{\mathbb R}\\
                              ({\h n},{\h x})&\mapsto \frac{1}{2}\|{\h x}-{\h n}-{\bar{\h x}}\|^2+\iota_{\cal S}({\h x}).
\end{array}
\end{eqnarray*}

We claim that  Assumption S. \ref{tran} holds with $Y=Z={\mathbb R}^m$,
$Q={\mathbb R}^m\times {\cal M}$ and $({\h y}_0,{\h z}_0)=({\bar{\h n}},{\bar{\h x}})$.
By noting that
\begin{eqnarray*} ({\h n},{\bf 0})\in N_{Q}({\bar{\h n}},{\bar{\h x}}) \Rightarrow {\h n} ={\bf 0}.\end{eqnarray*}

\noindent Second, $\rho({\h n},{\h x})$ is partly smooth at $({\bar{\h n}},{\bar{\h x}})$ relative to
${\mathbb R}^m\times{\cal M}$ since the set $\cal S$ is partly smooth at the point ${\bar{\h x}}$ relative to ${\cal M}$. Since $\partial \rho_{\bar{\h n}}({\bar{\h x}})=-{\bar{\h n}}+ N_{\cal S} ({\bar{\h x}})$
and
${\bar{\h n}}\in{\text{\rm ri}}N_{\cal S} ({\bar{\h x}})$, it ensures that $0\in {\rm ri}\partial \rho_{\bar{\h n}}({\bar{\h x}})$.
On the other hand,
we have the following inequality
\begin{eqnarray*}{\h x}\in{\cal M}\; \mbox{and}\;{\h x}\mbox{ near}\;{\bar{\h x}}\Rightarrow \frac{1}{2}\|{\h x}-{\bar{\h n}}-{\bar{\h x}}\|^2\ge\frac{1}{2}\|{\bar{\h n}}\|^2+\frac{1}{4}\|{\h x}-{\bar{\h x}}\|^2,\end{eqnarray*}
due to Lemma S. \ref{min}.
It implies that
${\bar{\h x}}$ is a strong critical point of $\rho_{\bar{\h n}}(\cdot)$ (setting $\delta=1/4$ in Definition S. \ref{stronglocal}).
 Then, using Theorem S. \ref{strongcrpa}, there exists a  mapping $\psi:{\mathbb R}^m\rightarrow{\cal M}$ such that
 $\psi({\bar{\h n}})={\bar{\h x}}$ and
 $\psi({\h n})$ is a strong critical point of $\rho_{\h n}(\cdot)$ relative to ${\cal M}$ near ${\bar{\h x}}$.
 It implies that
 $0\in{\text{\rm ri}}\partial \rho_{\h n}(\psi({\h n}))$.
 It further leads to
 ${\h n}+{\bar{\h x}}-\psi({\h n})\in {\text{\rm ri}} N_{\cal S}(\psi({\h n}))$.
 $P_{\cal S}({\h n}+{\bar{\h x}})=\psi({\h n}) =P_{\cal M}({\h n}+{\bar{\h x}})$
 due to $\psi({\h n})\in {\cal M}$.
\end{proof}

Based on Theorem S. \ref{spj}, we can further prove the following result which
we  relax the condition  in \cite[Theorem 4.1]{LW04} of ``the set $\cal S$ is
${\cal C}^p$-partly smooth ($p\ge 2$)" to
``the set $\cal S$ is partly smooth  (i.e., $\iota_{\cal S}$ be ${\cal C}^1$)" (see Definition S.\ref{PSset}).

\begin{theorem}\label{finiteiden}
Consider a set ${\cal S}$
  be partly smooth at the point ${\h x}$ relative the manifold ${\cal M}$,
  and prox-regular there. If the normal vector ${\bar{\h n}}$ is in ${\text{\rm ri}}N_{\cal S}({\bar{\h x}})$ and
  the sequence $\{{\h x}_k\}$ and $\{{\h d}_k\}$ satisfy
  \begin{eqnarray*}{\h x}_k \to{\bar{\h x}},\;{\h d}_k\to {\bar{\h n}},\; \mbox{and}\; {\text{\rm dist}}({\h d}_k, N_{\cal S}({\h x}_k))\to 0,\end{eqnarray*}
  then
  \begin{eqnarray*} {\h x}_k\in{\cal M} \;\;\mbox{for\;all\;large \;}k.\end{eqnarray*}
\end{theorem}
\begin{proof}
The proof follows the routines of \cite[Theorem 4.1]{LW04} by using Lemma S. \ref{proj} and Theorem S. \ref{spj}.
\end{proof}
The following theorem  relax the condition  in \cite[Theorem 4.3]{LW04} of ``the set $\cal S$ is
${\cal C}^p$-partly smooth ($p\ge 2$)" to
``the set $\cal S$ is partly smooth (i.e., $\iota_{\cal S}$ be ${\cal C}^1$)" (see Definition S.\ref{PSset}).
\begin{theorem}\label{identifycon}
 Consider a ${\cal C}^1$ function $f$. Let the set ${\cal S}$
  be partly smooth at the point ${\bar{\h x}}$ relative the manifold ${\cal M}$,
  and prox-regular there.
  Suppose ${\h x}_k\to{\bar{\h x}}$ and $-\nabla f({\bar{\h x}})\in{\text{\rm ri}}{\cal N}_{\cal S}({\bar{\h x}})$.
  Then, ${\h x}_k\in{\cal M}$ for all large $k$ if and only if ${\text{\rm dist}}(-\nabla f({\h x}_k),N_{\cal S}({\h x}_k))\to 0$.
 \end{theorem}
\begin{proof}
The proof follows the routines of \cite[Theorem 4.3]{LW04} by using Theorem S. \ref{finiteiden} for the ``sufficient direction".
\end{proof}

Now, we are in the stage to conclude the main result.
The following theorem  relaxes the condition  in \cite[Theorem 5.3]{LW04} of ``the function $f$ be ${\cal C}^p$-partly smooth ($p\ge 2$)"
 to
``the function $f$ be partly smooth (i.e., $f$ be ${\cal C}^1$)".

 \begin{theorem}\label{identify}
Let the function $f$ be partly smooth at the point ${\bar{\h x}}$ relative to the manifold ${\cal M}$,
and prox-regular there, with ${\bf 0}\in{\text{\rm ri}}\partial f({\bar{\h x}})$.
Suppose ${\h x}_k\to {\bar{\h x}}$ and $f({\h x}_k)\to f({\bar{\h x}})$. Then,
$${\h x}_k\in{\cal M}\; \mbox{for all large}\; k$$
if and only if
$${\text{\rm dist}}({\bf 0},\partial f({\h x}_k))\to 0.$$
\end{theorem}
\begin{proof} First, we show that if $f$  is partly smooth at a point ${\bar{\h x}}$ relative to a manifold ${\cal M}$,
${\rm epi}f$ is partly smooth at ${\bar{\h z}}=({\bar{\h x}},f({\bar{\h x}}))$ relative to the manifold ${\widehat{{\cal M}}}=\{({\h x},f({\h x})):{\h x}\in{\cal M}\}$.
We define
\begin{eqnarray*}\begin{array}{rl}
h: {\mathbb R}^{m+1}& \mapsto {\mathbb R}\\
({\h x},r)
&\mapsto f(\Phi({\h x},r))-r,
\end{array}
\end{eqnarray*}
where the function $\Phi({\h x},r)={\h x}$.
Thus, ${\rm epi}f =\{({\h x},r)|f(\Phi({\h x},r))-r\le 0\}.$ Applying the chain rule \cite{Lewis02}, sum rule \cite{Lewis02} and level set rule \cite{Lewis02}, we show that
${\rm epi}f$ is partly smooth at $({\bar{\h x}},f({\bar{\h x}}))$ relative to the manifold ${\widehat{{\cal M}}}$.
To show the chain rule \cite[Theorem 4.2]{Lewis02} applicable,  we only need to check $\Phi$ is transversal to ${\cal M}$, i.e.,
\begin{eqnarray*}{\text{Ker}}(\nabla \Phi({\h z})^*) \cap N_{\cal M} (\Phi ({\h z}))=\{{\bf 0}\},\end{eqnarray*}
where ${\h z}=({\h x},r)$.
\noindent Since $(\nabla \Phi({\h z}))^*=\left(\begin{array}{c} I_m\\ {\bf 0}_{1\times m}\end{array}\right)$,
then ${\text{Ker}}(\nabla \Phi({\h z})^*)=\{{\bf 0}\}$. Thus, the chain rule holds.
Therefore, ${\rm epi}f$ is partly smooth at $({\bar{\h x}},f({\bar{\h x}}))$ relative to the manifold ${\widehat{{\cal M}}}$
from that $f$  is partly smooth at a point ${\bar{\h x}}$ relative to a manifold ${\cal M}$.
Second, using \cite[Theorem 3.5]{PR96}, ${\rm epi}f$ is prox-regularity at ${\bar{\h z}}$.
 Define ${\cal S}={\text{epi}}f$, and define
\begin{eqnarray*} \begin{array}{ll}
\varphi:&{\mathbb R}^m\times {\mathbb R}\mapsto {\mathbb R}\\
&({\h x},r)\mapsto r.
\end{array}
\end{eqnarray*}
Let ${\h z}_k=({\h x}_k,f({\h x}_k))$ and
$-\nabla \varphi({\bar{\h z}})=({\bf 0},-1)\in {\text{\rm ri}}N_{\cal S}({\bar{\h z}})$.
Then, it follows from Theorem S. \ref{identifycon}, we have
\begin{eqnarray*}&&{\h x}_k\in{\cal M}\;\mbox{for\;all\; large\;} k\nn\\
&&\Leftrightarrow{\h z}_k\in\widehat{{{\cal M}}}\;\mbox{for\;all\; large\;} k\nn\\
&&\Leftrightarrow {\text{dist}}(-\nabla \varphi({\h z}_k),N_{\cal S}({\h x}_k,f({\h x}_k)))\to 0\nn\\
&&\Leftrightarrow {\text{dist}}(({\bf 0},-1),N_{\cal S}({\h x}_k,f({\h x}_k)))\to 0\nn\\
&&\Leftrightarrow {\text{dist}}({\bf 0},\partial f({\h x}_k))\to 0.
\end{eqnarray*}
\end{proof}	